\documentclass[a4paper,11pt]{amsart}

\usepackage{amssymb,xspace,amscd,yhmath,mathdots,txfonts,undertilde, 
mathrsfs}

\usepackage[matrix,arrow,curve]{xy}\CompileMatrices

\numberwithin{equation}{section}

\theoremstyle{plain}

\newtheorem{thm}{Theorem}[section]

\newtheorem{lem}[thm]{Lemma}

\newtheorem{cor}[thm]{Corollary}

\newtheorem{prop}[thm]{Proposition}

\theoremstyle{definition}

\newtheorem{exam}[thm]{Example}

\newtheorem{dfn}[thm]{Definition}
\newtheorem{defn}[thm]{Definition}

\DeclareMathOperator{\supp}{\mathrm{supp}}

\DeclareMathOperator{\im}{{\mathrm{Im}}}
\DeclareMathOperator{\re}{{\mathrm{Re}}}
\DeclareMathOperator{\Zt}{{\mathscr{Z}}}
\DeclareMathOperator{\Ot}{{\mathscr{O}}}
\DeclareMathOperator{\Otm}{{\mathscr{O}_M}}
\DeclareMathOperator{\Jtm}{\mathcal{J}_M}
\DeclareMathOperator{\Jt}{\mathcal{J}}

\title[Non completely solvable ...]{Non completely solvable systems\\ of
   complex first order 
PDE's}

\author[C.D.Hill]{C. Denson Hill}
 \address{C. D. Hill:       Department of Mathematics,\\
        Stony Brook University\\
        Stony Brook NY 11794 (USA)}
\email{dhilll@math.sunysb.edu}

\author[M.~Nacinovich]{Mauro Nacinovich}
\address{M.\ Nacinovich:
Dipartimento di Matematica\\ II Universit\`a di Roma
``Tor Ver\-ga\-ta''\\ Via della Ricerca Scientifica\\ 00133 Roma
(Italy)}
\email{nacinovi@mat.uniroma2.it}

\begin{document}
\maketitle


\section*{Introduction}
\subsection*{History and motivation} 
Hans Lewy in \cite{L57} and Louis Niremberg in \cite{Nir74}
gave two fundamental results in the theory of linear partial
differential equations. The first showed that a non homogeneous
equation for a first order partial differential operator with
complex valued
real analytic coefficients, but $\mathcal{C}^{\infty}$-smooth
right hand side, may, in general, have no local weak solution.
The second, that a homogeneous equation 
for a first order partial differential with complex valued smooth
coefficients may have no non constant weak local solutions.
Both results were formulated and proved for partial differential
operators in $\mathbb{R}^3$. 
A fuller understanding of \cite{L57} opened different directions
of investigation (see e.g. \cite{AFN81, AH72, Ho61, NT70, NT71}),
especially from the two points of view of p.d.e. theory and of
the analysis of $CR$ manifolds.\par
Nirenberg's example was especially relevant to the problem of
embedding $CR$ manifold into complex manifolds. From this point
of view, there have been two types of results. The
Nirenberg
example means that pseudoconvex
three dimensional $CR$ hypersurfaces
cannot be locally $CR$-embedded. However 
the existence of
sufficiently many independent solutions of the tangential 
Cauchy-Riemann equations was shown to hold for 
pseudoconvex higher dimensional $CR$ hypersurfaces
(see e.g. \cite{Ak87, BdM74, C94, K82a, K82b, K82c}), and
some general results were also obtained
in terms of Lie algebras of vector fields (see e.g. \cite{BR87, HN99}).
In the opposite direction, the counterexample of \cite{Nir74} was 
extended to $CR$ hypersurfaces with degenerate or non degenerate
Lorentzian signature (see e.g. \cite{H88, H91, J86, JT82, JT83} 
\par
The results above were all obtained for the case of $CR$ hypersurfaces.
For higher codimension, a crucial invariant is the
scalar Levi form, which is parametrized by the characteristic codirections
of the tangential Cauchy-Riemann complex. The first result 
in higher codimension on 
the absence of the Poincar\'e lemma at the place $q$
when some non degenerate scalar Levi form has $q$ positive eigenvalues
was first proved in \cite{AFN81}. In \cite{HN06} this result was extended
to some cases where the scalar Levi form is allowed to degenerate. 
Much less is known about the $CR$-embedding of manifolds of higher
$CR$ codimension. In \cite{mez93} some results of \cite{JT82} are
extended under some supplementary conditions of the Cauchy-Riemann
distribution. We also cite some partial results in \cite{BHN02,
BHN03, HN93}. \par
Here we want to reconsider some of these questions, also in the more general
framework of general distributions of complex vector fields 
of \cite{AHNP}.
\subsection*{Contents of the paper}
Let $M$ be a smooth paracompact manifold of dimension $m$, and let
$L_1,\hdots,L_n$ be smooth complex vector fields on $M$. In local coordinates
each $L_j$ can be written as
\begin{equation}
  \label{eq:a1}
  L_j=L_j(x,D)={\sum}_{i=1}^ma_{j,i}(x)\frac{\partial}{\partial{x}_i},
\end{equation}
with coefficients $a_{j,i}$ which are assumed to be complex valued and
$\mathcal{C}^{\infty}$-smooth.
We are interested in considering local solutions of the 
homogeneous system
\begin{equation}
  \label{eq:a2}
  L_ju=0,\quad \text{for}\; j=1,\hdots,n.
\end{equation}
When $n>1$, since every local distribution solution $u$ of \eqref{eq:a2}
also satisfies \begin{equation*}
[L_{j_1},L_{j_2}]u=
(L_{j_1}L_{j_2}-L_{j_2}L_{j_1})u=0, \;\hdots,\;
[L_{j_1},[L_{j_2},[\hdots,L_{j_r}]]]u=0
\end{equation*}
for all sequence
$j_1,j_2,\hdots,j_r$ with  $1\leq j_1,j_2,\hdots,j_r\leq{n}$,
it is not restrictive to require that $L_1,\hdots,L_n$ satisfy the
formal Cartan integrability conditions, i.e. that all commutators
$[L_{j_1},L_{j_2}]$ are linear combinations, with smooth coefficients,
of $L_1,\hdots,L_n$. \par
When this condition is satisfied, and
$L_1,\hdots,L_n$ define linearly independent tangent vectors on
a neighborhood $U_0$ of a point $p_0\in{M}$, there are at most
$m-n$ solutions $u_1,\hdots,u_{m-n}$
of \eqref{eq:a2}, with $du_1(p_0),\hdots,du_{m-n}(p_0)$ linearly
independent. In fact this is always the case when the $L_j$'s
have coefficients that are real analytic in some coordinate
neighborhood of $p_0$. Nirenberg's result in \cite{Nir74} shows
that in general this is not true in the $\mathcal{C}^{\infty}$ case 
if $n=1$, $m=3$. A small perturbation
of a vector field for which \eqref{eq:a2} has two analytically independent
solutions changes to a vector field for which all local solutions of
\eqref{eq:a2} are constant. \par
In \S{\ref{sec:b}}
we show how this result extends to the case
where $n=1$, but $m$ is allowed to be any integer larger or equal
to $3$. Namely, we show that the smooth complex vector fields
for which \eqref{eq:a2} admits non locally constant solutions near
some point of $M$ form a small nowhere dense set of first Baire
category in the Fr\'echet space of complex vector fields on $M$.
\par
This generalization of \cite{Nir74} was already given in \cite{mez93},
and our main goal is to extend in fact the results of \cite{JT82}
to the case of higher $CR$ codimension. 
We show that in general, given a smooth manifold $M$ and
any locally $CR$-embeddable
Lorentzian $CR$ structure on $M$, and a point $p_0\in{M}$,
there is a new Lorentzian $CR$ structure, 
which is defined on a neighborhood
of $p_0$ in $M$, and agrees to infinite order with the original one at
$p_0$, 
which is not locally $CR$-embeddable. 
We also show that the corresponding system \eqref{eq:a2} is not
completely integrable in the class $\mathcal{C}^1$.
\par
In \S\ref{sec:d} we collect the notions on $CR$ manifolds that will
be employed throughout the rest of the paper.
In \S\ref{sec:3x}, \S\ref{sec4}, \S\ref{sh}, \S\ref{sec:4} 
we prove the analog of the result
of \S\ref{sec:b} for overdetermined systems 
by adaptations of the arguments therein. 
The results
are weaker than those obtained for a scalar p.d.e. 
In fact, our constructions involve perturbations of an original
system which, to keep formal integrability, employ either 
functions that are constant with respect to some variables, or,
in \S\ref{sh}, special morphisms of $CR$ manifolds, and, 
in the more special cases of \S\ref{s:5}, 
analytic objects, called
$CR$-divisors. In general, we obtain
new overdetermined systems which are only defined in small
coordinate neighborhoods.
In \S\ref{s6} we prove that we can globally define a new
$CR$ structure on the Lorentzian real quadric $Q$ in $\mathbb{CP}^{\nuup}$
which is not locally $CR$-embeddable at all points of a hyperplane section.
\par
In \S\ref{sh} we also describe the $CR$ complexes and show in \S\ref{sec:62}
how the technique used in the rest of the paper can be also
employed to give proofs of the non validity of the Poincar\'e  lemma
different from those of \cite{AFN81, AH72, HN06}.

\section{Homogeneous equations with no nontrivial solutions}
\label{sec:b}
In this section we prove  
a generalization  to dimensions $\geq{3}$ of a remarkable theorem 
of Nirenberg 
about local homogeneous solutions
to a single homogeneous linear partial differential equation
having smooth variable complex coefficients
(\cite{Nir74}, see also  \cite{JT82, 
  mez93}).
\par\smallskip
Here and in the following sections, $M$ will denote a smooth paracompact
real manifold of dimension $m$.\par
We denote by $\mathfrak{X}^{\mathbb{C}}(M)$ the Fr\'echet space
of all $\mathcal{C}^{\infty}$ complex vector fields on $M$.
Note that $\mathfrak{X}^{\mathbb{C}}(M)$ includes also real
vector fields on $M$. 
When $M$ is an open set in $\mathbb{R}^m$, each 
$L\in\mathfrak{X}^{\mathbb{C}}(M)$ can be written as
\begin{equation*}
  L=a_1(x)\frac{\partial}{\partial{x}_1}+a_2(x)\frac{\partial}{\partial{x}_2}
+\cdots +a_m(x)\frac{\partial}{\partial{x}_m},
\end{equation*}
where $x=(x_1,x_2,\hdots,x_m)$, and the coefficients $a_j(x)\in
\mathcal{C}^{\infty}(M)$ are (in general) complex valued.\par
\begin{thm}\label{thm:b1}
  Let $M$ be a smooth manifold of dimension $m\geq{3}$. Then
the set $\mathfrak{E}$ 
of $L\in\mathfrak{X}^{\mathbb{C}}(M)$ for which there
exists a non empty open subset $U$ of $M$,
an $\epsilon>0$,   and a solution $u\in\mathcal{C}^{1+\epsilon}(U)$
of $Lu=0$ on $U$ with $du(p)\neq{0}$ for at least one $p\in{U}$,
is a nowhere dense set of first (thin) Baire category.
\end{thm}
In other words: the set of all $L$ on $M$ having the property
that any $u$ with H\"older continuous first derivatives, which
is a local solution to $Lu=0$, in any neighborhood of any point,
must be constant, is a dense set of the second (thick) Baire category.
\par
First we prove a Lemma.
\begin{lem}\label{lem:b2}
Let $M$ be a smooth Riemannian manifold of dimension $m\geq{3}$, and
let $L_0\in\mathfrak{X}^{\mathbb{C}}(M)$. Then for every 
point $p_0\in{M}$, $h\in\mathbb{N}$, and $\epsilon>0$ we can find 
$L\in\mathfrak{X}^{\mathbb{C}}(M)$ with
\begin{gather}
  \|L-L_0\|_{h,\, M}<\epsilon \quad\text{on}\quad{M},\;\\
\intertext{such that}
\label{eq:b2}
u\in\mathcal{C}^{1}(U),\;
U^{\mathrm{open}}\ni{p_0},\; Lu=0\;\text{on}\; U\;
\Longrightarrow \; du(p_0)=0.
\end{gather}
\end{lem}
\begin{proof}
We can argue on a small coordinate patch $\Omega$
about $p_0$, and then, substituting
$L_0$ by another vector field $L_0$ sufficiently close in the $h$-norm,
we can
assume that
the coefficients of $L_0$ are real analytic in the coordinates in $\Omega$,
and that
\begin{equation*}
  L_0(p),\;\bar{L}_0(p),\;[L_0,\bar{L}_0](p) \;
\text{are linearly independent in $\mathbb{C}T_pM$,$\forall{p}\in\Omega$}.
\end{equation*}
Let $k=m-2$.
By the real analyticity assumption, using the
Cauchy-Kowalevski theorem, and by shrinking $\Omega$ if needed,
we can find
$k+1$ complex valued real analytic $z_0,z_1,\hdots,z_m$ on $\Omega$ with
\begin{equation*}
  L_0z_i=0,\;\text{for $i=0,\hdots,k$},\quad dz_0\wedge d\bar{z}_0
\wedge dz_1\wedge\cdots\wedge dz_k
\neq{0} \quad \text{on}\;\Omega.
\end{equation*}
Let $x_i=\mathrm{Re}\,{z}_i$, $y_i=\mathrm{Im}\,{z}_i$.
We can also arrange
that $x_0,y_0,x_1,\hdots,x_k$ are real coordinates in $\Omega$ 
centered at $p_0$, and that $y_i(p_0)=0$,
$dy_i(p_0)=0$ for $i=1,\hdots,k$. This preparation yields a local
$CR$-embedding of $\Omega$ as a $CR$ submanifold of $CR$ dimension $1$
and $CR$ codimension $k$ in $\mathbb{C}^{k+1}$, given by
\begin{equation*}\begin{aligned}
 & y_i=h_i(z_0,x)\;\text{for}\; i=1,\hdots,k,\\
&\text{with}\;
x=(x_1,\hdots,x_k), \;\text{and}\; h_i=O(z_0\bar{z}_0+|x|^2).
\end{aligned}
\end{equation*}
After multiplication by a nowhere zero function, we can take
$L_0$ of the form
\begin{equation*}
  L_0=\frac{\partial}{\partial\bar{z}_0}+{\sum}_{i=1}^k{a_i}
\frac{\partial}{\partial{x}_i},\quad a_i\in\mathcal{C}^{\infty}(\Omega).
\end{equation*}
The condition that $[L_0,\bar{L}_0](p_0)\neq{0}$ implies that
$\partial^2h_i/\partial{z}_0\partial{\bar{z}_0}\neq{0}$ at $p_0$
for some index $i$. Moreover,
we note that we obtain new solutions of the homogeneous equation
$L_0u=0$ by taking for $u$ any holomorphic function of
$z_0,z_1,\hdots,z_k$. This allows us to use 
biholomorphic transformations
to obtain that
\begin{equation}\tag{$*$} \label{eq:bs}
  \text{the real Hessian of}\;  h_1(z_0,x)\;\text{is positive definite
in $\Omega$}.
\end{equation}
In this way, the sets $\Omega_r=\{p\in\Omega\mid \mathrm{Im}\,z_1<r\}$,
for $r>0$, form a fundamental system of open neighborhoods of
$p_0$ in $M$.
Set
\begin{equation*}
  M_{\tau}=\{p\in\Omega\mid z_1=\tau\},\;\text{for}\; \tau\in\mathbb{C}.
\end{equation*}
Then, by \eqref{eq:bs}, $M_0=\{p_0\}$, and
 there is an open 
connected neighborhood $\omega$ of $0$ in
$\mathbb{C}$ and a smooth real curve 
$\mathrm{Im}\,\tau=\phi(\mathrm{Re}\,\tau)$ in
$\omega$, passing through $0$,
with the properties
\begin{align}
  \tag{$i$}  & M_\tau\subset\Omega \;\text{if $\tau\in\omega$},\\
\tag{$ii$} & M_{\tau}=\emptyset \;\;
\text{if $\mathrm{Im}\,\tau<\phi(\mathrm{Re}\,\tau)$},\\ \label{eq:biii}
\tag{$iii$} & M_{\tau}=\{\text{a point}\}\;
\text{if $\mathrm{Im}\,\tau=\phi(\mathrm{Re}\,\tau)$},\\
\tag{$iv$} & M_{\tau}\simeq{S}^k\;
\text{if $\mathrm{Im}\,\tau>\phi(\mathrm{Re}\,\tau)$}.
\end{align}
Let $\{D_{\nu}\}$ be a sequence of pairwise disjoint 
closed discs in
\begin{displaymath}
 \omega^+=
\{\tau\in\omega\mid \mathrm{Im}\,\tau>\phi(\mathrm{Re}\,\tau)\},
\end{displaymath}
with centers and radii converging to $0$ for $\nu\to\infty$.
For a suitable $r_0>0$, all sets
\begin{displaymath}
  \omega^+_r=\{\tau\in\omega\mid \phi(\mathrm{Re}\,\tau)<\mathrm{Im}\,\tau
<r\}
\end{displaymath}
are connected, for $0<r<r_0$. Set
\begin{displaymath}
 \omega'_r=\omega^+_r\setminus{\bigcup}_{\nu}D_{\nu},\quad
 \Omega'_r=\{p\in\Omega\mid z_1(p)\in\omega'_r\}.
\end{displaymath}
Let $u$ be a $\mathcal{C}^1$ solution of $L_0(u)=0$ on
$\Omega'_r$,
 for some $0<r<r_0$. For each $\tau\in\omega'_r$ we define
\begin{equation*}
  F(\tau)=\int_{M_{\tau}}u \,dz_0\wedge dz_2\wedge\cdots\wedge dz_k.
\end{equation*}
We claim that $F$ is holomorphic in $\omega'_r$. 
Let indeed $\kappa$ be an arbitrary smooth simple closed curve in
$\omega_r$. Then ${\bigcup}_{\tau\in\kappa}M_{\tau}$ is the boundary
of a domain $N_{\kappa}$ in $\Omega$, that is diffeomorphic to
the Cartesian product of a $2$-disc and a $(k-1)$-ball, and
\begin{equation*}\begin{aligned}
\oint_{\kappa}F(\tau)d\tau&=
  \oint_{\kappa}d\tau\int_{M_{\tau}}\!\! u \,dz_0\wedge dz_2\wedge\cdots
\wedge dz_k 
=
\pm\int_{\partial{N}_{\kappa}}\! \! u \,dz_0\wedge dz_1\wedge
dz_2\wedge\cdots\wedge dz_k\\
&=
\pm\int_{N_{\tau}}du\wedge dz_0\wedge
dz_1\wedge\cdots\wedge dz_k=0,
\end{aligned}
\end{equation*}
because (see~\cite{hl56})
\begin{equation*}
  du\wedge{dz}_0\wedge{dz}_1\wedge\cdots\wedge{dz}_k=
(L_0u)\, d\bar{z}_0\wedge{dz}_0\wedge{dz}_1\wedge\cdots\wedge{dz}_k=0.
\end{equation*}
By Morera's theorem, $F$ is holomorphic on $\omega'_r$.
Moreover, $F$ extends to a continuous function on the closure of
$\omega_r'$ in $\omega\cap\{\mathrm{Im}\,\tau<r\}$, 
that equals $0$ for $\mathrm{Im}\,\tau=
\phi(\mathrm{Re}\,
\tau)$ because of \eqref{eq:biii}. It follows that
$F(\tau)=0$ for $\tau\in\omega'_r$. \par
For each $\nu\in\mathbb{N}$,
 let $Q_{\nu}=\{p\in\Omega\mid z_1(p)\in{D}_{\nu}\} $.
We fix smooth 
functions
$\psi_i$ in $\Omega$, such that $\psi_i dz_0\wedge d\bar{z}_0\wedge
\cdots\wedge{d}z_k$ is, for each $i=0,1,\hdots,k$, 
a non-negative real regular measure, with
\begin{equation*}
  \mathrm{supp}\,\psi_i={\bigcup}_{j=0}^{\infty}Q_{i+j(k+1)}
\end{equation*}
and such that, for
\begin{equation*}
  L=L_0+\psi_0\frac{\partial}{\partial{z}_0}+{\sum}_{i=1}^k
\psi_i\frac{\partial}{\partial{x}_i}
\end{equation*}
we have $\|L-L_0\|_h<\epsilon$.
Assume now that $u\in\mathcal{C}^{1}(\Omega_r)$
satisfies $Lu=0$. 
Hence, for all $\nu$ sufficiently large, $Q_{\nu}\subset\Omega_r$, and
\begin{equation*}\begin{aligned}
0&=\pm\oint_{\partial{D}_{\nu}}d\tau{\int}_{M_{\tau}}\!\! u\,
{dz}_0\wedge{dz}_2\wedge\cdots\wedge{dz}_k=
\int_{\partial{Q}_{\nu}}u\, dz_0\wedge{d}z_1\wedge\cdots\wedge
dz_k
\\ & 
=
\! \int_{Q_{\nu}}\!\!(L_0u)dz_0\wedge d\bar{z}_0
\wedge dz_1\wedge\cdots\wedge dz_k 
=\! \int_{Q_{\nu}}\!\!((L_0-L)u)dz_0\wedge d\bar{z}_0
\wedge dz_1\wedge\cdots\wedge dz_k
\end{aligned}
\end{equation*}
implies, by the mean value theorem, that, for all large $i\in\mathbb{N}$,
there are points $p_i,p'_i\in{Q}_j$ such that, for large $j$,
\begin{equation*}\begin{cases}
   \re\dfrac{\partial{u}(p_{j(1+k)})}{\partial{z}_0}=0,\;\;
\im\dfrac{\partial{u}(p'_{j(1+k)})}{\partial{z}_0}=0,
\\[8pt]
\re\dfrac{\partial{u}(p_{i+j(k+1)})}{\partial{x}_i}=0,\;\;
\im\dfrac{\partial{u}(p'_{i+j(k+1)})}{\partial{x}_i}=0,
\; \text{for $i=1,\hdots,k$}.
\end{cases}
\end{equation*}
By passing to the limit, as $p_j\to{p}_0$, 
we obtain that
\begin{equation*}
 \frac{\partial{u}(p_0)}{\partial{z}_0}=0,\;\;
\frac{\partial{u}(p_0)}{\partial{x}_i}=0,
\; \text{for $i=1,\hdots,k$}. 
\end{equation*}
Together with $Lu(p_0)=0$, this yields $du(p_0)=0$.
\end{proof}
\begin{proof}[Proof of Theorem \ref{thm:b1}]
We fix a Riemannian metric on $M$, so that we can compute the
length of vectors and covectors and the $\mathcal{C}^{h}$-norms
of functions defined on subsets of $M$. 
Then we have seminorms which endow $\mathfrak{X}(M)$ with a
Fr\'echet space topology, and we may discuss Baire category.
Let $\{U_{\nu}\}_{\nu\in\mathbb{N}}$ 
be a countable basis of non empty open subsets of $M$,
and for each $\nu\in\mathbb{N}$ fix a point  $p_{\nu}\in{U}_{\nu}$.
For $h\in\mathbb{N}$ we
define $\mathfrak{E}(\nu,h)$ to be the closure
in $\mathfrak{X}^{\mathbb{C}}(M)$ of the set of
$L$ such that
\begin{equation}\label{eq:b3}
  \exists u\in\mathcal{C}^{1+\tfrac{1}{h}}(U_{\nu})\;\text{with}\;
  \begin{cases}
    Lu=0\quad\text{on}\; U_{\nu},\\
\|u\|_{{1+\tfrac{1}{h}},\, U_{\nu}}\leq{h},\\
|du(p_{\nu})|\geq \frac{1}{h}.
  \end{cases}
\end{equation}
The set $\mathfrak{E}(\nu,h)$ has an
empty interior. 
This can be proved by contradiction. If some
$\mathfrak{E}(\nu,h)$ had an interior point, by Lemma \ref{lem:b2}
it would contain an interior point
$L$ satisfying \eqref{eq:b2} with $p_0=p_{\nu}$.
By definition, there is a sequence $\{L_j\}_{j\in\mathbb{N}}$
with $L_j\to{L}$ in $\mathfrak{X}^{\mathbb{C}}(M)$ for $j\to\infty$
such that for each $j$, there is 
$u_j\in\mathcal{C}^{1+\tfrac{1}{h}}(U_{\nu})$ with
$L_ju_j=0$ on $U_{\nu}$, 
$\|u_j\|_{1+\frac{1}{h},\,U_{\nu}}\leq{h}$ and
$|du(p_{\nu})|\geq\frac{1}{h}$.
By the Ascoli-Arzel\`a theorem, passing to a subsequence
we can assume that $u_j\to{u}\in\mathcal{C}^{1+\frac{1}{h}}(U_{\nu})$,
uniformly with their first derivatives on every compact neighborhood
of $p_{\nu}$ in $U_{\nu}$. Then $Lu=0$ on $U_{\nu}$, and
$|du(p_\nu)|\geq\frac{1}{h}>0$ contradicts \eqref{eq:b2}.
 \par
Therefore the union ${\bigcup}_{\nu,h}\mathfrak{E}(\nu,h)$ is 
a countable union of closed subsets having empty interior,
hence of first Baire category. Then also $\mathfrak{E}$ is
of first Baire category, because
$\mathfrak{E}\subset{\bigcup}_{\nu,h}\mathfrak{E}(\nu,h)$.
This completes the proof of the Theorem.
\end{proof}
\section{Involutive systems and $CR$ manifolds}\label{sec:d}
To extend the result of \S\ref{sec:b} 
to overdetermined systems of homogeneous first order p.d.e.'s,
we will develop ideas from \cite{H88, H91, J86}. 
In this section we begin by
describing the general framework. In the following, $M$ will denote
a $\mathcal{C}^{\infty}$-smooth manifold of real dimension $m$.
\subsection{Generalized complex distributions and $CR$ structures}
Let ${\Zt}$ be a generalized
distribution of smooth complex vector fields
on $M$. This means that ${\Zt}$ defines, for each open subset
$U$ of $M$ a, 
$\mathcal{C}^{\infty}(U)$ submodule
of $\mathfrak{X}^{\mathbb{C}}(U)$, in such a way that the assignment
$U\to{\Zt}(U)$ is a sheaf:
\begin{enumerate}
\item If $U^{\mathrm{open}}\subset{V}^{\mathrm{open}}\subset{M}$, 
then $Z|_{U}\in{\Zt}(U)$ for
all $Z\in{\Zt}(V)$;
\item If $\{U_{\nu}\}$ is a family of open subsets of $M$, a smooth 
complex vector field
$Z$, defined on ${\bigcup}_{\nu}U_{\nu}$, belongs to
$\Zt({\bigcup}_{\nu}U_{\nu})$ if and only if $Z|_{U_\nu}\in\Zt(U_{\nu})$
for all~$\nu$.
\end{enumerate}
\par
Our main interest in the sequel will be focused on the \textsl{local}
solutions 
to the homogeneous
system
\begin{equation}
  \label{eq:d1}
  Zu=0,\quad\forall Z\in{\Zt}.
\end{equation}
It is therefore natural to assume
in the following that
${\Zt}{}$ is \emph{involutive}, or \emph{formally integrable}. 
This means that
\begin{equation}
  \label{eq:d2}
  [Z_1,Z_2]\in{\Zt}(U),\;\forall Z_1,Z_2\in{\Zt}(U),
\quad\forall{U}^{\mathrm{open}}\subset{M}.
\end{equation}
Since ${\Zt}$ is a fine sheaf, every germ $Z_{(p)}$ of
${\Zt}$ at a point $p\in{M}$ is the restriction  of a
global $Z\in{\Zt}(M)$. Thus we can for simplicity
utilize global sections $Z\in{\Zt}(M)$ in most
of the discussion below.\par
For each point $p\in{M}$, we consider the set
\begin{equation}
  \label{eq:d3}
  \mathbf{Z}_pM=\{Z(p)\mid Z\in{\Zt}(M)\}\subset\mathbb{C}T_pM.
\end{equation}
If the dimension of the $\mathbb{C}$-linear space $ \mathbf{Z}_pM$
is constant, we say that ${\Zt}$ is a \emph{distribution} of
complex vector fields. 
\begin{dfn} A \emph{$CR$ structure} on $M$ is the datum of an 
involutive distribution
${\Zt}$ of smooth complex vector fields with ${\Zt}\cap\overline{{\Zt}}=
{{}}{0}$.
\par
The constant dimension $n$ of $\mathbf{Z}_pM$ is its
$CR$ dimension, and $k=m-2n$ 
its $CR$ codimension. We call the pair $(n,k)$ the \emph{type} of 
the $CR$ manifold~$M$.
\end{dfn}
In the case where $\Zt$ is a $CR$ structure on $M$, we 
write sometimes $T^{0,1}M$ for~$\mathbf{Z}M$.\par\smallskip
When $M$ is a real smooth submanifold of a complex manifold $\mathbb{X}$,
we consider on $M$ the generalized distribution
\begin{equation*}
  \Zt(U)=\{Z\in\mathfrak{X}^{\mathbb{C}}(U)\mid Z_p\in{T}^{0,1}_p\mathbb{X},\;
\forall p\in{U}\},
\end{equation*}
where ${T}^{0,1}\mathbb{X}$ is the bundle of 
anti-holomorphic complex tangent vectors to 
$\mathbb{X}$. Then $\Zt$ is involutive and ${\Zt}\cap\overline{{\Zt}}=
{{}}{0}$. When $\Zt$ has constant rank, 
$\Zt$ defines
a $CR$ structure on $M$, for which we say that $M$ is a \emph{$CR$-submanifold}
of $\mathbb{X}$. \par
Let $1\leq{a}\leq\infty$. A 
\emph{complex $CR$-immersion} of class $\mathcal{C}^a$ of
$M$ 
is a  $\mathcal{C}^a$-smooth immersion
$\phiup:M\to\mathbb{X}$ of $M$ into a complex manifold $\mathbb{X}$ 
with $d\phiup(\mathbf{Z}_pM)\subset{T}^{0,1}_{\phiup(p)}\mathbb{X}$
for all $p\in{U}$.

\par\smallskip
For any open set $U$ of $M$ we set
\begin{equation}
\Otm(U)=\{u\in \mathcal{C}^1(U)\mid
Zu=0,\;\forall Z\in{\Zt}(M)\}. 
\end{equation}
The assignment $U^{open}\to\Otm(U)$
defines a sheaf of
rings of germs of complex valued differentiable functions on $M$.
\subsection{The differential  ideal and complete integrability}
Let $\varOmega^*_M={\bigoplus}_{0\leq{p}\leq{m}}\varOmega^p_M$ 
be the sheaf of germs of alternated 
smooth differential forms on $M$. We associate to $\Zt$ the 
\emph{differential ideal}
\begin{equation}
  \label{eq:33} \Jtm=
{\bigoplus}_{p\geq{1}}\Jt^p_M 
\;\;\text{with}\;\;
\Jt_M^p=
\{\etaup\in\varOmega^p_M\mid \etaup|_{\Zt(M)}=0\}.
\end{equation}
This is a graded ideal sheaf of $\varOmega^*_M$, generated by its
elements of degree $1$. 
Being interested in the local solutions to \eqref{eq:d1},
we can assume that $\Jtm$ is \emph{complete} and that 
$\Zt$ is the \emph{characteristic system} of 
$\Jtm$, 
i.e. that \begin{align*}
\Zt(U)&=\{Z\in\mathfrak{X}^{\mathbb{C}}(M)\mid Z\rfloor\Jtm(U)\subset\Jtm(U)\}
\\
&=\{Z\in\mathfrak{X}^{\mathbb{C}}(U)\mid \etaup(Z)=0,\;\forall
\etaup\in\Jt_M^1(U)\}, \quad\forall{U}^{\text{open}}\subset{M}.
\end{align*}
If $\Zt$ is a distribution, it is the characteristic system of its differential
ideal. The pointwise evaluation of the elements of $\Jt_M^1$
yields in this case a smooth subbundle $\mathbf{Z}^0M$ of $\mathbb{C}T^*M$,
given by
\begin{equation}
  \label{eq:d4a}
  \mathbf{Z}^0M={\bigsqcup}_{p\in{M}} \mathbf{Z}^0_pM,\;\;\text{with}\;\;
 \mathbf{Z}^0_pM=\{\etaup\in\mathbb{C}T^*_pM\mid \etaup(Z)=0
\;\forall Z\in{\Zt}(M)\}.
\end{equation}
In general, \eqref{eq:d4a} defines a subset of the complexified tangent bundle
of $M$. 
\begin{dfn} Let ${\Zt}$ be a generalized
distribution of smooth complex vector fields on $M$ and $p_0\in{M}$. 
We say that ${\Zt}$ is 
\emph{completely integrable} at $p_0$ if
\begin{equation}
  \label{eq:d5}
  \forall \etaup\in\mathbf{Z}_{p_0}^0M
\quad 
\exists u\in\Otm_{(p_0)}\;\;
\text{with}\;\; du(p_0)=\etaup.
\end{equation}
\end{dfn} 
This means that \eqref{eq:d1} has at $p_0$ the largest number of
differentially independent local solutions that is permitted by the rank
of ${\Zt}$.
\subsection{The case of $CR$ manifolds}
Let ${\Zt}$ be a $CR$ structure of type $(n,k)$ on $M$.
Complete integrability at $p_0\in{M}$ 
is equivalent to the existence
of a complex $CR$-immersion of class $\mathcal{C}^1$ of 
an open neighborhood $U$ of $p_0$ 
into $\mathbb{C}^{n+k}$. 
\par
The question of the regularity of complex
$CR$-immersions seems in general a rather
delicate open problem (see e.g.
\cite{MM94}). Note that any $\mathcal{C}^1$-immersion is in fact 
$\mathcal{C}^{\infty}$-smooth
when $M$ satisfies suitable pseudo-concavity assumptions
(see \cite{AHNP}).\par
For $\mathcal{C}^{\infty}$-smooth complex local $CR$-immersions 
we introduce a special notation.
\begin{defn}
A $CR$-chart on $M$ is the datum of an open subset $U$  and 
of $n+k$ smooth $CR$ functions 
$z_1,\hdots,z_{n+k}\in\Otm(U)\cap\mathcal{C}^{\infty}(U)$, 
such that
\begin{equation*}
  dz_1(p)\wedge\cdots\wedge{dz}_{n+k}(p)\neq{0},\;\forall{p}\in{U}.
\end{equation*}
\end{defn}
Clearly $\phiup(p)=(z_1(p),\hdots,z_{n+k}(p))$ provides in this case
a $\mathcal{C}^{\infty}$-smooth $CR$-immersion of $U$ in $\mathbb{C}^{n+k}$.
\par
The functions 
$z_1,\hdots,z_{n+k}$ of a $CR$-chart are 
not independent complex coordinates
when $k>0$. For each point $p_0$ of $U$ there are indeed $k$ real valued
functions $\rho_1,\hdots,\rho_k$, 
defined and $\mathcal{C}^{\infty}$ 
on an open neighborhood $G$ of $\phiup(p_0)$ in
$\mathbb{C}^{n+k}$, with $\rho_i(z_1,\hdots,z_{n+k})=0$ on a neighborhood
of $p_0$, and 
$\partial\rho_1(\phiup(p_0))\wedge\cdots\wedge\partial\rho_k(\phiup(p_0)) 
\neq{0}$. 
\begin{defn}
We say that a $CR$ manifold $M$ is \emph{locally $CR$-embeddable} if 
the open subsets $U$ of its $CR$-charts make a covering.  
\end{defn}
\par
Locally $CR$-embeddable $CR$ manifolds can be abstractly defined 
as ringed spaces, using the structure sheaf 
$\Ot^{\infty}_M={\Ot}_M\cap\mathscr{C}^{\infty}$
of the germs of its smooth $CR$ functions. 
\begin{lem}\label{lem25}
Let $M$ be a $CR$ manifold of type $(n,k)$ and $p_0\in{M}$. 
Then we can find an open neighborhood $U$ of $p_0$ in $M$
and a new $CR$ structure on $M$ which is locally $CR$-embeddable
and agrees to infinite order with the original one at $p_0$.
\end{lem}
\begin{proof} Let $\Zt$ be the $CR$ structure on $M$. 
It suffices to consider smooth functions $z_1,\hdots,z_{\nuup}$
which are defined on a neighborhood of $p_0$, satisfy
$Zz_j=0^{\infty}$ at $p_0$, and have $dz_1(p_0)\wedge\cdots\wedge
dz_{\nuup}(p_0)\neq{0}$.  
To prove the existence of such functions, we observe that it is
always possible to find a smooth coordinate chart 
$(U,x_1,\hdots,x_m)$ centered at $p_0$ such that $\Zt$ is generated
in $U$ by vector fields of the form
\begin{equation*}
  Z_i=\dfrac{\partial}{\partial{x}_i}+i\dfrac{\partial}{\partial{x}_{i+n}}+
{\sum}_{j=n+1}^ma_j(x)\dfrac{\partial}{\partial{x}_j},\quad
\text{with $a_j(x)=O(|x|)$}.
\end{equation*}
Let $L_i=\dfrac{\partial}{\partial{x}_i}+i\dfrac{\partial}{\partial{x}_{i+n}}$,
and $R_i={\sum}_{j=n+1}^ma_j(x)\dfrac{\partial}{\partial{x}_j}$. \par
We denote by $\mathfrak{m}$ the maximal ideal of 
the local
ring $\mathbb{C}\{\{x_1,\hdots,x_m\}\}$ of formal power series of
$x_1,\hdots,x_m$. We obtain
formal power series solution 
to \eqref{eq:d1} by constructing by recurrence 
sequences $\{f_h\}_{h\geq{0}}\subset\mathbb{C}\{\{x_1,\hdots,x_m\}\}$ 
which  solve the equations
\begin{equation}\tag{$*$} \label{star}
  \begin{cases}
f_h\in\mathfrak{m}^h,\\
    L_jf_1\in\mathfrak{m}, &\text{for $j=1,\hdots,n$},\\
L_jf_{h+1}+R_jf_h\in\mathfrak{m}^{h+1},&\text{for $j=1,\hdots,n$}.
  \end{cases}
\end{equation}
We observe that, taking $f_1$ equal to 
$x_i+ix_{i+n}$ for $i=1,\hdots,n$, or to 
$x_{2n+i}$, for $i=1,\hdots,k$, we obtain $\nuup$ independent solutions
of $L_if_1=0$ for $1\leq{i}\leq{n}$. \par
Assume now that $d\geq{1}$ and 
$f_d\in\mathfrak{m}^d$ satisfies
\begin{equation*}
  L_if_d+R_if_{d-1}\in\mathfrak{m}^d,\quad\text{for}\quad 1\leq{i}\leq{n}.
\end{equation*}
The integrability conditions yield
$[Z_i,Z_j]=0$ for $1\leq{i,j}\leq{n}$. Hence we obtain
\begin{equation}\tag{$**$} \label{sstar}
  0=[Z_i,Z_j]f_d=-L_iR_jf_d+L_jR_if_d+[R_i,R_j]f_d.
\end{equation}
We have $R_if_d\in\mathfrak{m}^d$, and hence there is a 
polynomial $g_{i,d}\in\mathbb{C}[x_1,\hdots,x_m]$, 
homogeneous of degree $d$, such that
$R_if_d-g_{i,d}\in\mathfrak{m}^{d+1}$. 
Since $[R_i,R_j]f_d\in\mathfrak{m}^{d+1}$,  
we obtain from \eqref{sstar} that
$L_ig_{j,d}=L_jg_{i,d}$ for all $1\leq{i,j}\leq{n}$
and therefore there is a polynomial $f_{d+1}\in\mathbb{C}[x_1,\hdots,x_m]$, 
homogeneous of
degree $d+1$, such that $L_if_{d+1}=g_{i,d}$ for $i=1,\hdots,n$.
The series ${\sum}f_d$ of the terms of a sequence
$\{f_d\}$ solving \eqref{star} is a formal power series solution
of \eqref{eq:d1}. \par
In particular, we can find solutions
$\{z_1\},\hdots,\{z_{\nuup}\}\in\mathbb{C}\{\{x_1,\hdots,x_m\}\}$ 
to \eqref{eq:d1} with $d\{z_i\}(0)=dx_i(0)+idx_{i+n}(0)$ for $i=1,\hdots,n$
and $d\{z_i\}(0)=dx_{n+i}(0)$ for $i=n+1,\hdots,\nuup$. It suffices then
to take smooth functions $z_1,\hdots,z_{\nuup}$ having Taylor series
$\{z_1\},\hdots,\{z_{\nuup}\}$ at $0$.
\end{proof}

\subsection{Characteristic bundle and Levi form}
\cite{N84}
The underlying real distribution and the characteristic bundle
of ${\Zt}$ are:
\begin{align}\label{eq:d6}
  &\mathcal{H}=\{\mathrm{Re}\,Z\mid Z\in{\Zt}\},\quad
\text{i.e}\quad 
\mathcal{H}(U)=\{\mathrm{Re}\,Z\mid Z\in{\Zt}(U)\},
\;\forall U^{\mathrm{open}}\subset{M},
\\\label{eq:d7}
& H^0M =\{\xi\in{T}^*M\mid \xi(X)=0,\;\forall X\in
\mathcal{H}(M)\}.
\end{align}
To each characteristic covector $\xi_0\in{H}^0_{p_0}M$
we associate a
Hermitian symmetric form on $\mathbf{Z}_{p_0}M$, by
\begin{equation}
  \label{eq:d8}
  \mathfrak{L}_{\xi_0}(Z_1,Z_2)=i\xi_0([Z_1,\bar{Z}_2]),
\quad\forall Z_1,Z_2\in{\Zt}(M).
\end{equation}
In fact a straightforward verification shows that the value of the
right hand side of \eqref{eq:d8} only depends  
on $Z_1(p_0),Z_2(p_0)\in\mathbf{Z}_{p_0}M$. \par
Moreover, $\mathfrak{L}_{\xi_0}(Z_1,Z_2)=0$ if 
one of the two vector fields is real valued
on a neighborhood of $p_0$. 
Thus $\mathfrak{L}_{\xi_0}$ 
defines a Hermitian symmetric
form on the quotient of $\mathbf{Z}_{p_0}M$ by 
the subspace
$\mathbf{N}_{p_0}M=\{Z(p_0)\mid Z\in{\Zt}(M)\cap\overline{{\Zt}(M)}\}$,
consisting of the values at $p_0$ of the complex multiples of the
real vector fields
in ${\Zt}(M)$. 
Set
\begin{equation}
  \label{eq:d9}
  \check{\mathbf{Z}}_{p_0}M=
\mathbf{Z}_{p_0}M/\mathbf{N}_{p_0}M.
\end{equation}
If $\xi_0\in{H}^0_{p_0}M$, then \eqref{eq:d8} defines a Hermitian symmetric form
$\mathbf{L}_{\xi_0}$ on $\check{\mathbf{Z}}_{p_0}M$, that we call
the \emph{Levi form} of ${\Zt}$ at $\xi_0$.
\begin{dfn} Let $p_0\in{M}$ and $\xi_0\in{H}^0_{p_0}M$.
We say that ${\Zt}$ is $q$-pseudoconvex at $\xi_0$
if $\mathbf{L}_{\xi_0}$ 
is nondegenerate 
 and has 
exactly $q$ positive eigenvalues on $\check{\mathbf{Z}}_{p_0}M$. \par
If ${\Zt}$ is $1$-pseudoconvex at 
some $\xi_0\in{H}^0_{p_0}M$, we say that
${\Zt}$ is \emph{Lorentzian} at $p_0$.
\end{dfn}
If ${\Zt}(M)$ is generated by a single
vector field $L$ near $p_0$, the condition of being Lorentzian at
$p_0$ means that $L(p_0)$, $\bar{L}(p_0)$, and $[L,\bar{L}](p_0)$ 
are linearly independent in $\mathbb{C}T_{p_0}M$. 
\subsection{Reduction of complete integrability
to the case of $CR$ manifolds}
When $\mathbf{N}_pM$ has constant dimension on a neighborhood $U$ of
$p_0\in{M}$, then the real vector fields in $\Zt(U)$ 
define an involutive 
distribution $\mathscr{N}$ 
of \textit{real} vector fields on $U$. By the Frobenius theorem, there is 
an open  neighborhood $W$ of $p_0$ in $U$ and a smooth fibration
$\piup:W\to{B}$ of $W$ such that $B$ is a smooth manifold and
the fibers of $\piup$ are integral submanifolds of $\mathscr{N}$.
One easily proves 
\begin{lem}
 There is a $CR$ structure $\Zt'$ on $B$ such that for every $p\in{W}$ 
 we have $\Ot_{\! M,(p)}=\piup^*\Ot_{\! B,(\piup(p))}$, and $\Zt$ is completely
 integrable at $p\in{W}$ if and only if $\Zt'$ is completely integrable at 
 $\piup(p)$.
\end{lem}
\section{Involutive systems which are not completely
integrable at $p_0$}\label{sec:3x}
In this section, we 
give a weak generalization 
 of the results of \S\ref{sec:b} 
to Lorentzian $CR$ manifolds $M$ with arbitrary $CR$-codimension $k\geq{1}$
and $CR$-dimension $n\geq{2}$. We recall that 
$m=\dim_{\mathbb{R}}M=2n+k$, and we set $\nuup=n+k$.\par
We closely follow the arguments of \S\ref{sec:b}.\par
Assume that $M$ is locally $CR$-embeddable and Lorentzian at $p_0$.  
 Then there is
a $CR$-chart  $(U,z_1,\hdots,z_{\nuup})$ centered at $p_0$, with 
$dz^i(p_0)$ real for $i=n+1,\hdots,\nuup$, and 
\begin{equation}
  \label{eq:51a}
  \im{z}_{\nuup}+z_{\nuup}\bar{z}_{\nuup}+
{\sum}_{i=1}^{n-1}z_j\bar{z}_j={\sum}_{i=n}^{\nuup-1}z_j\bar{z}_j
+O(|z|^3)\quad\text{on $U$}.
\end{equation}
By shrinking, we get 
${\sum}_{i=n}^{\nuup-1}
z_i\bar{z}_i\geq \tfrac{1}{2}\im{z}_{\nuup} \;$ on $U$.\par
We consider the map $\piup:U\ni{p}\to w=(z_{1}(p),\hdots,z_{n-1}(p),z_{\nuup}(p))
\in\mathbb{C}^n$. By a further shrinking, we can assume 
that there is an open ball $B\subset\mathbb{C}^n$, centered at $0$, 
such that
\begin{list}{-}{}
\item $\piup(U)=\bar{\omega}$, with $B\setminus\omega$ 
strictly convex, 
and $\partial\omega\cap{B}$ smooth;
\item if $\im{\tauup}\geq{0}$, 
then $\{w\in{B}\mid \im{w}_n=\tauup\}\subset\omega$; 
\item for all $w\in\omega$ the set $M_{w}=\piup^{-1}(w)$ is diffeomorphic
to the sphere $S^k$;
\item for $w\in\partial\omega\cap{B}$ the set $M_{w}=\piup^{-1}(w)$ is a point. 
\end{list}
As in \S\ref{sec:b}, we have:
\begin{lem}\label{lem51}
If $u\in\Otm(U)$, then
\begin{equation}
  \label{eq:52}
  F(w)=\int_{M_w}u\, dz_n\wedge\cdots\wedge{dz}_{\nuup-1}=0,\;\;
\forall w\in\omega.
\end{equation}
\end{lem}
\begin{proof}
We prove first that $F$ is holomorphic on $\omega$. \par
Fix any polycylinder $D=D_1\times\cdots\times{D}_n$ in $\omega$,
with $D_i=\{\tau\in\mathbb{C}\mid |\tau-\tau_i|\leq\epsilon_i\}$.
For $1\leq{j}\leq{n}$ we set $\partial_j(D)=\{w\in{D}\mid |w_j-\tau_j|=\epsilon_j\}$,
$\gammaup_j=\dfrac{\partial}{\partial{\bar{w}_j}}\rfloor (d\bar{w}_1\wedge\cdots
\wedge{d}\bar{w}_n)$
and consider the integral 
\begin{align*}
  &\oint_{\partial_jD}F(w)dw_1\wedge\cdots\wedge{dw}_n\wedge\gammaup_j
=
\oint_{\partial_jD}dw_1\wedge\cdots\wedge{dw}_n\wedge\gammaup_j
\int_{M_w}u \,dz_n\wedge\cdots\wedge{dz}_{\nuup-1}.
\end{align*}
Let $N_i=\piup^{-1}(\partial_iD)$ and $N=\piup^{-1}(D)$. We have
$\partial{N}={\sum}_{i=1}^n\pm{N}_i$. Moreover, the form 
$u\, dz_1\wedge\cdots\wedge{dz_{\nuup}}\wedge\piup^*\gammaup_j$ is zero on
$N_i$ for $i\neq{j}$. Thus we obtain:
\begin{align*}
& \oint_{\partial_jD}F(w)dw_1\wedge\cdots\wedge{dw}_n
\wedge\gammaup_j
=\pm \int_{N_j}u\, dz_1\wedge\cdots\wedge{dz_{\nuup}}\wedge 
\piup^*\gammaup_j\\
&\quad 
={\sum}_{i=1}^n
\int_{{N}_{i}}\pm \,u \, dz_1\wedge \cdots{dz}_{\nuup}\wedge \piup^*\gammaup_j 
=\pm
\int_{{N}}du\wedge dz_1\wedge \cdots{dz}_{\nuup}\wedge \piup^*\gammaup_j
=0
\end{align*}
because $du\in\langle dz_1,\hdots, dz_{\nuup}\rangle$ by the assumption
that $u\in\Otm(U)$. This equality, valid for all closed polycylinder $D$
in $\omega$ and all $1\leq{j}\leq{n}$, implies that $F$ is holomorphic in
$\omega$. 
Clearly $F(w)\to{0}$ when $w\to\partial\omega\cap{B}$, because 
$M_{w_0}$ is a point for \mbox{$w_0\in\partial\omega\cap{B}$},
and hence $F=0$ on $\omega$ by Holmgren's uniqueness theorem,
since $\bar\partial$ has constant coefficients in $\mathbb{C}^n$.
\end{proof}

Let $\psi$ be a smooth function with compact support in $\mathbb{C}$,
and set
\begin{equation*}
  \hat{\psi}(\tauup)=\frac{1}{2\pi{i}}\iint \dfrac{\psi(\zetaup)
d\zetaup\wedge{d}\bar{\zetaup}}{\zetaup-\tauup}.
\end{equation*}
Then $\dfrac{\partial\hat{\psi}}{\partial\bar{\tauup}}=\psi$ and
therefore \begin{equation*}
\psi^{\sharp}(\zetaup,\tauup)=
\bar{\zetaup}\psi(\tauup)d\bar{\tauup}+\hat{\psi}(\tauup)d\bar\zetaup
=\bar{\partial}(\bar{\zetaup}\,\hat{\psi}(\tauup))
\end{equation*}
is a $\bar\partial$-closed form in $\mathbb{C}^2$, with
\begin{equation*}
d\psi^{\sharp}=\dfrac{\partial{\hat{\psi}(\tauup)}}{\partial{\tauup}}d\tauup
\wedge d\bar{\zetaup}
+\bar{\zetaup}\dfrac{\partial\psi(\tauup)}{\partial\tauup}d\tauup\wedge{d}\bar\tauup
=d\tauup\wedge \dfrac{\partial}{\partial\tauup}\psiup^{\sharp}.
\end{equation*}

\begin{lem}\label{lem52} Let $U'\Subset{U}$. 
If $\psi_i$, for $i=1,\hdots,\nuup-1$, are smooth functions of a complex
variable $\tauup$, with $|\psi_i|$ sufficiently small.  
Then
\begin{equation}
  \label{eq:53}
  \thetaup_1=dz_1+\psi_1^{\sharp}(z_n,z_{\nuup}),\;\hdots,\;
\thetaup_{\nuup-1}=d{z}_{\nuup-1}+\psi_{\nuup-1}^{\sharp}(z_n,z_{\nuup}),
\; \thetaup_{\nuup}=dz_\nuup
\end{equation}
generate the involutive ideal sheaf $\Jt'_M$
of a $CR$ structure of type $(n,k)$ on $U'$.
\end{lem}
\begin{proof}
The ideal sheaf is generated on $U$ by $dz_1,\hdots,dz_{\nuup}$.
After shrinking, we can assume that $d{z}_1$,
$\hdots$, $d{z}_n$, $d\bar{z}_1$, $\hdots$, $d\bar{z}_n$ are linearly
independent on $U$.\par
Thus, by taking $|\psi_i|$ sufficiently small, we may keep 
$\thetaup_1,\hdots,\thetaup_{\nuup},
\bar{\thetaup}_1,\hdots,\bar{\thetaup}_n$ 
linearly independent in any neighborhood $U'$ of $p_0$ with 
$U'\Subset{U}$.
Moreover, since $d\psi_i^{\sharp}(z_n,z_{\nuup})\wedge{dz_{\nuup}}=0$,
for $1\leq{i}<\nuup$, 
we obtain
  \begin{equation*}
    (d\thetaup_i)\wedge\thetaup_1\wedge\cdots\wedge\thetaup_{\nuup}=
d\psi_i^{\sharp}(z_n,z_{\nuup})
\wedge{\thetaup}_1\wedge\cdots
\wedge\thetaup_{\nuup-1}\wedge{dz}_{\nuup}=0,\;\;\forall i=1,\hdots,\nuup-1.
\end{equation*}
This shows that the ideal sheaf $\Jt_M'$
generated by $\thetaup_1,\hdots,
\thetaup_{\nuup}$ is involutive and 
defines a $CR$ structure of type
$(n,k)$ on~$U'$.
\end{proof}
Let us fix a sequence of distinct complex numbers
$\{\tauup_j\}$, such that
\begin{list}{}{}
\item $\im\tauup_j>0$ for all $j$, 
\quad $\tauup_j\to{0}$, \quad $\{w_{n}=\tauup_j\}
\cap\omega\neq\emptyset$ for all $j$.
\end{list}  
For each $j$ we choose an open disk
$\Delta_j$  in $\mathbb{C}$, centered
at $\tauup_j$, and such that
$\bar{\Delta}_j\cap{\bigcup}_{i\neq{j}}\bar{\Delta_i}=\emptyset$. 
Provided the $\tau_j$'s are sufficiently close to $0$,
for each $j$ we can fix a point $w^{(j)}\in\omega$, with $w^{(j)}_n=\tauup_j$,
and $w^{(j)}\to{0}$, 
 and take the functions $\psi_j$ in Lemma\,\ref{lem52} in such a way that
\begin{gather*}
  \supp\psi_i={\bigcup}_{j=0}^{\infty}\bar{\Delta}_{i+j{\nuup}},\quad
\text{for}\quad i=1,\hdots,n,\\
c_{i+j{\nuup}}=\int_{A_{i+j{\nuup}}}\psi_i^{\sharp}(z_n,z_{\nuup})
\wedge{d}z_{n}\wedge\cdots\wedge{dz}_{\nuup-1} \wedge
dz_{\nuup}\;\;\text{is real and $>0$},
\end{gather*}
where ($e_1,\hdots,e_n$ is the canonical basis of $\mathbb{C}^n$)
\begin{equation*}
  A_j=\piup^{-1}(\{w^{(j)}+(\tauup-\tauup_j){e_n}\mid \tauup\in\Delta_j\}).
\end{equation*}
Let $u$ be a $CR$ function on an open neighborhood $V$ of $p_0$ in
$U$ for the structure defined by
\eqref{eq:53}. This means that $du_{(p)}\in\Jt'_{M\,(p)}$ for all
$p\in{V}$. 
Since $\Jtm$ and $\Jt_{M}'$ agree to infinite order 
outside ${\bigcup}_{j}\piup^{-1}(\{w\mid w_n\in\Delta_j\}$, and 
${\bigcup}_j\{w\in\omega\mid w_n\in\Delta_j\}$ does not disconnect
$\omega$, by the argument of Lemma\,\ref{lem51} 
we have \eqref{eq:52} for all $w$ in the complement in 
$\piup(V)\setminus{\bigcup}_{j}\{w\in\omega\mid w_n\in\Delta_j\}$.
Thus we obtain
\begin{align*}
  0=\pm \oint_{\tauup\in\partial\Delta_j}d\tauup\int_{M_{w^{(j)}+\tauup{e}_n}}
u dz_n\wedge\cdots\wedge{d}z_{\nuup-1}
&=\pm\int_{\partial{A}_j}u\, dz_n\wedge\cdots\wedge{dz}_{\nuup}\\
&=\pm\int_{A_j}du\wedge
dz_n\wedge\cdots\wedge{dz}_{\nuup}.
\end{align*}
This yields
\begin{equation*}
  \int_{A_i+j{\nuup}}\dfrac{\partial{u}}{\partial{z}_i}\,
\psi_i^{\sharp}\wedge dz_n\wedge\cdots\wedge{dz}_{\nuup}
=0,
\end{equation*}
where, to compute $\dfrac{\partial{u}}{\partial{z}_i}$, we consider any
$\mathcal{C}^1$-extension of $u$ as a function of 
the complex variables $z_1,\hdots,z_{\nu}$ for which $\bar\partial{u}=0$
at all points of $U$. Taking the limit, we observe that
\begin{equation*}
  c_{i+j{\nuup}}^{-1}\int_{A_i+j{\nuup}}\dfrac{\partial{u}}{\partial{z}_i}\,
\psi_i^{\sharp}\wedge dz_n\wedge\cdots\wedge{dz}_{\nuup}
\longrightarrow \dfrac{\partial{u}(p_0)}{\partial{z}_i}
\Longrightarrow \dfrac{\partial{u}(p_0)}{\partial{z}_i}=0 \;\; \forall
i=1,\hdots,\nuup-1,
\end{equation*}
which, together with \eqref{eq:d1} shows that $du(p_0)\in
\mathbb{C}\,dz_{\nuup}(p_0)$. \par
We have proved:
\begin{thm} Let $M$ be a $CR$ manifold of type $(n,k)$ and assume that
$M$ is Lorentzian at a point $p_0$. Then
we can find a new $CR$ structure of type $(n,k)$ on 
a neighborhood $U$ of $p_0$, which agrees with
the original one to infinite order at $p_0$, 
and a real codirection $\etaup_0\in{T}^*_{p_0}M$ such that,
if $\Zt$ is the distribution of $(0,1)$-vector fields for this
new structure, 
 all solutions
$u\in\mathcal{C}^1$ on a neighborhood of $p_0$ to the homogeneous
system \eqref{eq:d1} satisfy $du(p_0)\in\mathbb{C}\etaup_0$.   
\end{thm}
\begin{proof}
  Indeed, using Lemma\,\ref{lem25} we can always reduce to the case
in which $M$ is locally embeddable at $p_0$.
\end{proof}

\begin{cor}
We can find a new $CR$ structure of type $(n,k)$ on $U$, which agrees with
the original one to infinite order at $p_0$, and which is not $CR$-embeddable
at~$p_0$.  
\end{cor}
\section{Involutive systems whose solutions are critical at $p_0$}
\label{sec4}
In this section we improve the result of the previous section in the
case of a Lorentzian $CR$ manifold of the hypersurface type.\par
We assume that
$M$ has $CR$-dimension $n\geq{2}$ and $CR$-codimension ${1}$,
and is Lorentzian
and locally embeddable at $p_0\in{M}$.
We have $m=\dim_{\mathbb{R}}M=n+2$ and 
we set $\nuup=n+1$, 
 \par
We can fix 
a $CR$-chart  $(U,z_1,\hdots,z_{\nuup})$ centered at $p_0$, with 
\begin{equation}
  \label{eq:41a}
  \im{z}_{\nuup}+{\sum}_{i=2}^{\nuup}z_i\bar{z}_i
=z_1\bar{z}_1
+O(|z|^3)\quad\text{on $U$}.
\end{equation}
By shrinking, we get that $z_1\bar{z}_1
\geq\tfrac{1}{2}\im{z_{\nuup}}$
on $U$. 
Consider the map $\piup:U\ni{p}\to w=(z_{2}(p),\hdots,z_{\nuup}(p))
\in\mathbb{C}^n$. By a further shrinking, we can assume 
that there is an open ball $B\subset\mathbb{C}^n$, centered at $0$, 
such that
\begin{list}{-}{}
\item $\piup(U)=\bar{\omega}$, with $B\setminus\omega$ 
strictly convex, 
and $\partial\omega\cap{B}$ smooth;
\item if $\im{\tauup}\geq{0}$, 
then $\{w\in{B}\mid \im{w}_n=\tauup\}\subset\omega$; 
\item for all $w\in\omega$ the set $M_{w}=\piup^{-1}(w)$ is diffeomorphic
to the circle $S^1$;
\item for $w\in\partial\omega\cap{B}$ the set $M_{w}=\piup^{-1}(w)$ is a point. 
\end{list}
By repeating the proof of Lemma\,\ref{lem51}, we obtain
\begin{lem}\label{lem41}
If $u\in\Otm(U)$, then
\begin{equation}
  \label{eq:52}\qquad\qquad\qquad
  F(w)=\oint_{M_w}u\, dz_1=0,\;\;
\forall w\in\omega.\qquad\qquad\qquad \qed
\end{equation}
\end{lem}
Since $2\,z_1\bar{z}_1\geq\im{z}_{\nuup}$ on $U$,
for any smooth function $\psi$ of a complex variable
$\tauup$, with $\supp\psi\subset\{\im\tauup\geq{0}\}$, the function
$z_1^{-1}\psi(z_{\nuup})$ can be extended to a smooth function on $U$, vanishing to
infinite order on $\{z_1=0\}\cap{U}$. 
\begin{lem}\label{lem32a}
If $\psi_i$, for $i=1,\hdots,\nuup$ are smooth fuctions of a complex variable
$\tauup$, with support contained in $\{\im\tauup\geq{0}\}$, then
\begin{equation}
  \label{eq:52a}
  \thetaup_1=dz_1+z_1^{-1}{\psi_1({z_{\nuup}})}d{\bar{z}_{\nuup}},
\;\hdots,\; \thetaup_{\nuup}=
dz_{\nuup}+z_1^{-1}{\psi_{\nuup}({z_{\nuup}})}d{\bar{z}_{\nuup}}
\end{equation}
(the functions $z_1^{-1}{\psi_i({z_{\nuup}})}$ are put $=0$ for $z_1=0$) 
generate the ideal sheaf $\Jt_{U'}'$ of a $CR$ structure of type $(n,1)$
in a neighborhood $U'$ of $p_0$ in $U$, which agrees to infinite order with
the original one at $p_0$. 
\end{lem}
\begin{proof} By the condition 
on the supports, the functions 
$z_1^{-1}{\psi_i(z_{\nuup})}$ are smooth on $U$ and
vanish to infinite for $z_1=0$, and in particular at $p_0$. 
Thus $\thetaup_1$, $\hdots$, $\thetaup_\nuup$, $\bar{\thetaup}_1$,
$\hdots$, $\bar{\thetaup}_n$ yield a basis of $\mathbb{C}T_pM$ for $p$
in a suitable neighborhood $U'$ of $p_0$, and agree with
$dz_1, \hdots, dz_{\nuup}, d\bar{z}_1,\hdots,d\bar{z}_n$ to infinite order
at $p_0$. \par
We have moreover
\begin{equation*}
  d\thetaup_i=z_1^{-1}
\dfrac{\partial{\psi_i({z_{\nuup}})}}{\partial{{z_{\nuup}}}}d{z_{\nuup}}\wedge
d{\bar{z}_{\nuup}}-{z_1^{-2}}\psi_i({z_{\nuup}})dz_1\wedge{d}{\bar{z}_{\nuup}}.
\end{equation*}
Hence
\begin{equation*}
  d\thetaup_i\wedge\thetaup_1\wedge\cdots\wedge\thetaup_{\nuup}=
d\thetaup_i\wedge{dz_1}\wedge\cdots\wedge{d}z_{\nuup}=0
\end{equation*}
shows that $\Jt'_{U'}$ is involutive. The proof is complete.
\end{proof}
Let us fix a sequence of distinct complex numbers
$\{\tauup_j\}$, such that
\begin{list}{}{}
\item $\im\tauup_j>0$ for all $j$, 
\quad $\tauup_j\to{0}$, \quad $\{w_{n}=\tauup_j\}
\cap\omega\neq\emptyset$ for all $j$.
\end{list}  
For each $j$ we choose an open disk
$\Delta_j$  in $\mathbb{C}$, centered
at $\tauup_j$, and such that
$\bar{\Delta}_j\cap{\bigcup}_{i\neq{j}}\bar{\Delta_i}=\emptyset$. 
Provided the $\tau_j$'s are sufficiently close to $0$,
for each $j$ we can fix a point $w^{(j)}\in\omega$, with $w^{(j)}_n=\tauup_j$,
and $w^{(j)}\to{0}$, 
 and take the functions $\psi_j$ in Lemma\,\ref{lem32a} in such a way that
\begin{gather*}
  \supp\psi_i={\bigcup}_{j=0}^{\infty}\bar{\Delta}_{i+j{(\nuup+1)}},\quad
\text{for}\quad i=1,\hdots,\nuup,\\
c_{i+j{(\nuup+1)}}=\int_{A_{i+j({\nuup}+1)}}\!\!\!
z_1^{-1}\psi_i({z_{\nuup}})d{\bar{z}_{\nuup}}
\wedge{d}z_{1}\wedge
dz_{\nuup}\;\;\text{is real and $>0$},
\end{gather*}
where 
\begin{equation*}
  A_j=\piup^{-1}(\{w^{(j)}+(\tauup-\tauup_j){e_n}\mid \tauup\in\Delta_j\}).
\end{equation*}
Here we denoted by 
$e_1,\hdots,e_n$ the canonical basis of $\mathbb{C}^n$. \par
Let $u$ be a $CR$ function on an open neighborhood $V$ of $p_0$ in
$U'$ for the structure defined by
\eqref{eq:52a}. This means that $du_{(p)}\in\Jt'_{U'\,(p)}$ for all
$p\in{V}$. 
Since $\Jtm$ and $\Jt_{U'}'$ agree to infinite order on $\piup(U')$ 
outside ${\bigcup}_{i}\supp{\psi_i(w)}$, and this set
does not disconnect $U$, by the argument of Lemma\,\ref{lem51} 
we have \eqref{eq:52} for all $w$ in the complement in $\piup(V)$ of
${\bigcup}_{j}\{w\in\omega\mid w_n\in\Delta_j\}$.
Thus we obtain
\begin{equation*}
0=\pm \oint_{\tauup\in\partial\Delta_j}d\tauup\oint_{M_{w^{(j)}+\tauup{e}_n}}
\!\!\!
u \,dz_n
=\pm\int_{\partial{A}_j}u\, dz_1\wedge{dz}_{\nuup}
=\pm\int_{A_j}du\wedge
dz_1\wedge{dz}_{\nuup}.  
\end{equation*}
This yields
\begin{equation}\label{eq:444}
  I_{i+j(\nuup+1)}(u)=\int_{A_i+j{(\nuup+1)}}\dfrac{\partial{u}}{\partial{z}_i}\,
z_1^{-1}\wedge{d\bar{z}_{\nuup}}\wedge dz_1\wedge{dz}_{\nuup}
=0,
\end{equation}
where, to compute $\dfrac{\partial{u}}{\partial{z}_i}$, we consider any
$\mathcal{C}^1$-extension of $u$ as a function of 
the complex variables $z_1,\hdots,z_{\nu}$ for which $\bar\partial{u}=0$
at all points of $V$. When $j\to\infty$, 
$c_{i+j{(\nuup+1)}}^{-1}I_{i+j{(\nuup+1)}}
\to \dfrac{\partial{u}(p_0)}{\partial{z}_i}$.
Hence, from
\eqref{eq:444} 
we obtain that 
$\dfrac{\partial{u}(p_0)}{\partial{z}_i}=0$ for $1\leq{i}\leq\nuup$,
which, together with \eqref{eq:d1} shows that $du(p_0)=0$. \par
We have proved:
\begin{thm} \label{thm43}
If $M$ is a $CR$ manifold of type $(n,1)$ and is 
Lorentzian at $p_0\in{M}$, then
we can find a new $CR$ structure of type $(n,1)$ on 
an open neighborhood $U$ of $p_0$ in $M$, which agrees with
the original one to infinite order at $p_0$, 
such that,
if $\Zt$ is the distribution of $(0,1)$-vector fields for this
new structure, 
 all solutions
$u\in\mathcal{C}^1$ on a neighborhood of $p_0$ to the homogeneous
system \eqref{eq:d1} satisfy $du(p_0)=0$.
\end{thm}
\begin{proof}
We can apply the discussion above after reducing, by Lemma\,\ref{lem25},
to the case in which $M$ is locally $CR$-embeddable  at $p_0$.
\end{proof}
\section{The case of higher codimension}\label{sh}
In this section we extend the result of Theorem\,\ref{thm43}
to some $CR$ manifolds with $CR$ dimension and $CR$ codimension 
both greater than $1$. To this aim we will first 
recall some results on weak unique continuation and next consider 
morphisms of $CR$ manifolds.
\subsection{Minimal locally $CR$-embeddable 
$CR$ manifolds and unique continuation}
We recall that a $CR$ submanifold 
$M$ is \emph{minimal} at $p_0\in{M}$ if there is
no germ $(N,p_0)$  of $CR$ submanifold of $M$ at $p_0$, having 
the same $CR$ dimension, but smaller $CR$ codimension.
We have
\begin{lem} \label{lem:3.1}
Assume that $M$ is minimal and locally $CR$-embeddable at $p_0\in{M}$. \par
Let $(S,p_0)$ be a germ 
of a $CR$ submanifold of $M$, of type $(0,\nuup)$.
Then a germ $f\in\Ot_{M,(p_0)}$  of a 
$CR$ function at $p_0$, vanishing on $(S,p_0)$,
is equal to $0$. \par
If $M$ is minimal and locally $CR$-embeddable
at all points, then the $CR$ functions on $M$ satisfy the weak
unique continuation principle.
\end{lem}
\begin{proof}
In the first part of the proof, we can assume that $M$ is a generic
$CR$ submanifold of an open set in $\mathbb{C}^{\nuup}$. For
any open neighborhood $U$ of $p_0$ in $M$, 
there are an open neighborhood 
$U_0$ of $p_0$ in $U$, and an open wedge $W$ in $\mathbb{C}^{\nuup}$,
with edge $U_0$, such that, the restriction $u|_{U_0}$
of any 
$u\in{\Ot}_M(U)$  is the boundary value of
a holomorphic
function $\tilde{u}$, defined  
on  $W$ (see \cite{Tu88, Tu90, BR90}). 
Assume now that $u\in\Otm(U)$ vanishes on $S$.
Then $\tilde{u}=0$ by the edge of the wedge theorem (see \cite{PK87}),
and therefore $u=0$.
\par
The last statement follows by unique continuation 
for holomorphic functions 
on open subsets of $\mathbb{C}^{\nuup}$.
\end{proof}
\subsection{$CR$-maps with simple singularities}\label{sec:5xa}
Let $M,N$ be $CR$ manifolds.
A smooth map $\piup:M\to{N}$ is
$CR$ if $d\piup(T^{0,1}M)\subset{T}^{0,1}N$. We say that $\piup$ is
\begin{list}{-}{}
\item a $CR$-immersion if $\ker{d\piup}=0$ and 
$d\piup(T^{0,1}M)=d\piup(\mathbb{C}TM)\cap{T}^{0,1}N$;
\item a $CR$-submersion if $d\piup(T_pM)=T_{\piup(p)}N$ and
$d\piup(T^{0,1}_pM)={T}^{0,1}_{\piup(p)}N$, $\forall p\in{M}$;
\item a local $CR$-diffeomorphism if it is at the same time 
a $CR$-immersion and a $CR$-submersion.
\end{list} \par
Next we consider critical points of some $CR$-maps.\par
Let $k\geq{1}$ and $\piup:M\to{N}$ a $CR$-map, with 
$M$ of type $(n,k)$ and $N$ of type $(n,k-1)$.
\par
If $p_0\in{M}$ is not a critical point of $\piup$, then
$\piup$ is a $CR$-submersion near $p_0$.  
\par
Assume now that 
$p_0$ is a critical point of $\piup$, and $q_0=\piup(p_0)$ the corresponding
critical value. Then the rank of $d\piup(p_0)$ is less than $2n+k-1$.
Assume that it is exactly equal to $2n+k-2$.
Then the dual map $d\piup^*(p_0):T^*_{q_0}N\to{T}^*_{p_0}M$ is not injective,
and has a $1$-dimensional kernel. 
\begin{defn}
 If $\ker{d}\piup^*(p_0)\cap{H}^0_{q_0}N=\{0\}$,
we say that $\piup$ has at $p_0$ a \emph{$CR$-noncharacteristic}
singularity.
\end{defn}
Assume that this is the case and fix
$0\neq\eta_0\in\ker{d}\piup^*(p_0)$. Then there is 
$\eta'_0$, uniquely determined modulo $H^0_{q_0}N$, such that
$\eta_0+i\eta_0'\in\mathbf{Z}^0_{q_0}N$, and we obtain an element
$\xiup_0\in{H}^0_{p_0}M$, with
$0\neq\xiup_0=d\piup^*(p_0)(\eta_0')$.
\begin{defn} If we can choose $\eta_0'$ in such a way that
$\mathbf{L}_{\xiup_0}$ has $1$ positive and $n-1$ negative eigenvalues,
we say that $\piup$ has a \emph{Lorentzian $CR$-non characteristic
singularity} at $p_0$.   
\end{defn}
Assume now that $M$ and
$N$ are locally $CR$-embeddable at $p_0$, $q_0$, respectively, and that 
$\mathbf{L}_{\xiup_0}$ has $1$ positive and $(n-1)$ negative eigenvalues.
We set $\nuup=n+k$. 
We can choose $CR$-charts $(U;z_1,\hdots,z_{\nuup})$ of $M$, centered 
at $p_0$, and $(W;w_2,\hdots,w_{\nuup})$ of $N$, centered at $q_0$, 
with $\piup(U)\subset{W}$ and $z_j=\piup^*w_j$ for $j=2,\hdots,\nuup$,
such that $i\xiup_0=dz_{\nuup}(p_0)$, and 
\begin{equation*}
  \im{z}_{\nuup}=h(z) \quad\text{on $U$, with $h(z)=z_1\bar{z}_1-
{\sum}_{i=2}^{\nuup}z_i\bar{z}_i+O(|z|^3)$}.
\end{equation*}

\begin{lem} \label{lem516}
Let $D=\{p\in{U}\mid z_1(p)=0\}$. Then, there is
an open neighborhood $U'$ of $p_0$ in $U$, an open neighborhood
$\omega$ of $q_0$ in $N$, 
and an open domain $\omega_-$ in $\omega$, 
with $q_0\in\partial\omega_-$, such that
\begin{enumerate}
\item
 ${\omega}_-\subset \piup(U')\subset\omega$,\;\;
$\piup(D\cap{U}')\subset\partial\omega$ \;\; $\text{and
$\omega$ is strictly pseudoconcave at $q_0$}$;
\item $\piup:U'\to{N}$ is proper and, for $q\in\piup(U')$,
$\piup^{-1}(q)$ is either a point or is diffeomorphic to a 
circle. 
\end{enumerate}
\end{lem}
\begin{proof} 
Provided $U$ is sufficiently small, 
the restriction of $\piup$ to $D$ is a smooth diffeomorphism of $D$
onto a closed hypersurface $\piup(D)$ 
in an open neighborhood $\omega$ of $q_0$ in $N$. 
By further shrinking, we can assume that 
$A\setminus \piup(D)$ consists of two connected components 
$\omega_+$ and $\omega_-$ and that
$\omega_-\subset\piup(U)$.
\par
Since $\omega_-=\{\im{w_{\nu}}+{\sum}_{i=2}^{\nuup}w_i\bar{w}_i+O(|w|^3)>0\}$
near $q_0$, we have $q_0\in\partial\omega_-$ and $\omega_-$
strictly pseudoconcave at $q_0$. Moreover, by taking $U$ small,
we can assume that 
\begin{equation*}
  \im{z}_{\nuup}+\tfrac{3}{2}{\sum}_{i=2}^{\nuup}
z_i\bar{z}_i\geq \tfrac{1}{2} {z}_1\bar{z}_1 \quad\text{on $U$},
\end{equation*}
and therefore we obtain an $U'$ satisfying (1) and (2) by setting
$U'=U\cap\piup^{-1}(\omega)$ for a smaller neighborhood
$\omega$ of $q_0$ in $N$. 
\end{proof}
\subsection{Perturbation of the $CR$ structure of $M$}
We keep the notation of \S\ref{sec:5xa}, and we shall assume that
(1) and (2) of Lemma\,\ref{lem516} hold true with $U'=U$.
\begin{lem}\label{lemr54} 
Assume that $N$ is a minimal $CR$ manifold.
If $u$ is a $CR$ function on a connected open neighborhood $V$
of $p_0$ in $U$, then
\begin{equation}
  \label{eq:r1}
  g(q):=\oint_{\piup^{-1}(q)}u\, dz_1=0,\quad\forall q\in\piup({V}).
\end{equation}
\end{lem}
\begin{proof} First we note that $W=\piup(V)\cup(\omega\setminus\piup(U))$
is a neighborhood of $q_0$ in $N$. The function $g$, equal to the left
hand side of \eqref{eq:r1} for $w\in\piup(V)$ and $0$ on 
$W\setminus\piup(V)$ is continuous,
because  the fiber $\piup^{-1}(q)$ shrinks to a point when
$q\to \partial{\piup(V)}\cap\omega$. 
Since $\piup(V)$ is connected and its connected component in $W$ contains
an open subset where $g=0$, our contents follows by the
weak unique continuation principle (see Lemma\,\ref{lem:3.1})
if we show that $g$ is a $CR$ function on $W$. 
To this aim, it suffices to
show that
\begin{equation*}
  \int_N dg\wedge\eta=0,\quad\forall \eta\in\varOmega^{2n+k-2}_0(W)\cap\Jt_N^{n+k-1}(W),
\end{equation*}
where $\varOmega^{*}_0(W)$ means smooth exterior forms with compact support
in $W$. We note that 
$\piup^*\eta\in\varOmega^{2n+k-2}_0(V)\cap\Jt_M^{n+k-1}(W)$, because the map
$\piup$ is $CR$ and proper. Thus we obtain
\begin{equation*}
  \int_N dg\wedge\eta=\int_M du\wedge{dz_1}\wedge\piup^*\eta =0,
\end{equation*}
because $u$ is $CR$ on a neighborhood of the support of
${dz_1}\wedge\piup^*\eta\in\varOmega^{2n+k-2}_0(V)\cap\Jt_M^{n+k}(V)$. 
The proof is complete.
\end{proof}
By shrinking, we get
$2\,z_1\bar{z}_1\geq\im{z}_{\nuup}$ on $U$.
In particular, if $\psi$ is a smooth function of one complex variable
$\tauup$, with $\supp\psi\subset\{\im\tauup\geq{0}\}$, the function
$z_1^{-1}\psi(z_{\nuup})$ can be extended to a smooth function on $U$, vanishing to
infinite order on $\{z_1=0\}\cap{U}$. 
\begin{lem}\label{lemr32a}
If $\psi_i$, for $i=1,\hdots,\nuup$ are smooth fuctions of a complex variable
$\tauup$, with support contained in $\{\im\tauup\geq{0}\}$, then
\begin{equation}
  \label{eq:52a}
  \thetaup_1=dz_1+z_1^{-1}{\psi_1({z_{\nuup}})}d{\bar{z_{\nuup}}},
\;\hdots,\; \thetaup_{\nuup}=
dz_{\nuup}+z_1^{-1}{\psi_{\nuup}({z_{\nuup}})}d{\bar{z_{\nuup}}}
\end{equation}
(the functions $z_1^{-1}{\psi_i({z_{\nuup}})}$ are put $=0$ for $z_1=0$) 
generate the ideal sheaf $\Jt_{U'}'$ of a $CR$ structure of type $(n,k)$
in a neighborhood $U'$ of $p_0$ in $U$, which agree to infinite order with
the original one at $p_0$. 
\end{lem}
\begin{proof} By the condition 
on the supports, the functions 
$z_1^{-1}{\psi_i(z_{\nuup})}$ are smooth on $U$ and
vanishing to infinite for $z_1=0$, and in particular at $p_0$. 
Thus $\thetaup_1$, $\hdots$, $\thetaup_\nuup$, $\bar{\thetaup}_1$,
$\hdots$, $\bar{\thetaup}_n$ yield a basis of $\mathbb{C}T_pM$ for $p$
in a suitable neighborhood $U'$ of $p_0$, and agree with
$dz_1, \hdots, dz_{\nuup}, d\bar{z}_1,\hdots,d\bar{z}_n$ to infinite order
at $p_0$. 
We have moreover
\begin{equation*}
  d\thetaup_i=z_1^{-1}
\dfrac{\partial{\psi_i({z_{\nuup}})}}{\partial{{z_{\nuup}}}}d{z_{\nuup}}\wedge
d{\bar{z}_{\nuup}}-{z_1^{-2}}\psi_i({z_{\nuup}})dz_1\wedge{d}{\bar{z}_{\nuup}}.
\end{equation*}
Hence
\begin{equation*}
  d\thetaup_i\wedge\thetaup_1\wedge\cdots\wedge\thetaup_{\nuup}=
d\thetaup_i\wedge{dz_1}\wedge\cdots\wedge{d}z_{\nuup}=0
\end{equation*}
shows that $\Jt'_{U'}$ is involutive. The proof is complete.
\end{proof}
Let us fix a sequence of distinct complex numbers
$\{\tauup_j\}$, such that
\begin{list}{}{}
\item $\im\tauup_j>0$ for all $j$, 
\quad $\tauup_j\to{0}$, \quad $\{w_{n}=\tauup_j\}
\cap\omega\neq\emptyset$ for all $j$.
\end{list}  
For each $j$ we choose an open disk
$\Delta_j$  in $\mathbb{C}$, centered
at $\tauup_j$, in such a way that
$\bar{\Delta}_j\cap{\bigcup}_{i\neq{j}}\bar{\Delta_i}=\emptyset$. Next
we choose
balls $B_j$ in $\mathbb{C}^{\nuup-2}$ with 
\begin{equation*}
  K_j=\{q\in\omega\mid (w_2(q),\hdots,w_{\nuup-1}(q))\in{B}_j,\;\; w_{\nuup}\in
\bar{\Delta}_j\}\Subset\omega.
\end{equation*}
We set $\partial_0K_j=\{q\in{K}_j\mid w_\nuup(q)\in\partial\Delta_j\}$,
$A_j=\piup^{-1}(K_j)$. Note that the $A_j$'s are compact because
$\piup$ is proper.\par
Then
we take the functions $\psi_i$ in Lemma\,\ref{lemr32a} in such a way that,
for suitable forms $\eta_j\in\varOmega_0^{n-1}(B_j)$, we have 
\begin{gather*}
  \supp\psi_i={\bigcup}_{j=0}^{\infty}\bar{\Delta}_{i+j{(\nuup+1)}},\quad
\text{for}\quad i=1,\hdots,\nuup ,\\
 c_{i+j(\nuup+1)}=\int_{A_{i+j(\nuup+1)}}
z_1^{-1}\psi_i(z_{\nuup})
d\bar{z}_{\nuup}\wedge{d}z_1\wedge\cdots\wedge{d}z_{\nuup}\wedge\piup^*\eta_j
\;\;\text{is real and $>0$.}
\end{gather*}
\par
Let $u$ be a $CR$ function on a
connected open neighborhood $V$ of $p_0$ in
$U'$ for the structure defined by
\eqref{eq:52a}.
Since $\Jtm$ and $\Jt_{U'}'$ agree to infinite order on $\piup(U')$ 
outside $E=\{w_\nuup\in{\bigcup}_{i}\supp{\psi_i}\}$, and this set
does not disconnect $\piup(V)$, by the argument of Lemma\,\ref{lemr54} 
we have \eqref{eq:r1} for all $q$ in the complement in $\piup(V)$ of
$E$. 
Thus we obtain
\begin{align*}
  0&=\int_{\partial{K}_j}\left(\oint_{\piup^{-1}(q)}
\!\!\!
u \,dz_1\right){dw}_2\wedge\cdots\wedge{d}w_{\nuup}\wedge\eta_j
=\int_{\partial{A}_j}u\, dz_1\wedge\cdots\wedge{dz}_{\nuup}\wedge
\piup^*\eta_j\\
&=\int_{A_j}du\wedge dz_1\wedge\cdots\wedge{dz}_{\nuup}\wedge
\piup^*\eta_j,\qquad\text{yielding}
\end{align*}
\begin{equation}\label{eq:578}
  \int_{A_i+j{(\nuup+1)}}\dfrac{\partial{u}}{\partial{z}_i}\,
z_1^{-1}\psi_i(z_{\nuup})\,d\bar{z}_{\nuup}\wedge{d}z_1\wedge
\cdots\wedge{dz}_{\nuup}\wedge\piup^*\eta_{i+j{(\nuup+1)}}
=0,
\end{equation}
where, to compute $\dfrac{\partial{u}}{\partial{z}_i}$, we consider any
$\mathcal{C}^1$-extension of $u$ as a function of 
the complex variables $z_1,\hdots,z_{\nu}$ for which $\bar\partial{u}=0$
at all points of $V$. But
\begin{equation*}
  c_{i+j{(\nuup+1)}}^{-1}\int_{A_i+j{(\nuup+1)}}\dfrac{\partial{u}}{\partial{z}_i}\,
z_1^{-1}\psi_i(z_{\nuup})\,d\bar{z}_{\nuup}\wedge{d}z_1\wedge
\cdots\wedge{dz}_{\nuup}\wedge\piup^*\eta_{i+j{(\nuup+1)}}
\end{equation*}
converges to $\dfrac{\partial{u}(p_0)}{\partial{z}_i}$ when $j\to+\infty$. 
Therefore $\dfrac{\partial{u}(p_0)}{\partial{z}_i}=0$ for $i=1,\hdots,
\nuup$.
Together with \eqref{eq:d1}, this shows that $du(p_0)=0$. \par
We have proved:
\begin{thm} Let $M$, $N$, be $CR$ manifolds of types $(n,k)$ and $(n,k-1)$,
respectively, with $k\geq{1}$. Assume that $N$ is minimal and that
there is a $CR$ map
$\piup:M\to{N}$ having a Lorentzian $CR$-non characteristic singularity
at $p_0\in{M}$. Then
we can find a new $CR$ structure of type $(n,k)$ on 
an open neighborhood $U$ of $p_0$ in $M$, which agrees with
the original one to infinite order at $p_0$, 
such that,
if $\Zt$ is the distribution of $(0,1)$-vector fields for this
new structure, 
 all solutions
$u\in\mathcal{C}^1$ on a neighborhood of $p_0$ to the homogeneous
system \eqref{eq:d1} satisfy $du(p_0)=0$.
\end{thm}
\begin{proof}
The discussion above proves the theorem in the case where both
$M$ and $N$ are locally $CR$-embeddable at $p_0$ and at $q_0=\piup(p_0)$,
respectively. In general, we can reduce to this case by taking
formal power series solutions $\{z\}_1,\hdots,\{z\}_{\nuup}$ at $p_0$,
$\{w\}_2,\hdots,\{w\}_{\nuup}$ at $q_0$, to the homogeneous tangential
Cauchy-Riemann systems on $M$ and $N$, respectively, with
$\{z_j\}=\piup^*\{w\}_j$ for $j=2,\hdots,\nuup$. Then we take
smooth functions $w_2,\hdots,w_{\nuup}$ on $N$ having Taylor series
$\{w\}_2,\hdots,\{w\}_{\nuup}$ at $q_0$, define $z_j=\piup^*w_j$ for
$j=2,\hdots,\nuup$, and choose a smooth function $z_1$ on $M$ 
with Taylor series $\{z_1\}$ at $p_0$. By restricting to 
a suitable neighborhood $U$ of $p_0$ in $M$ and $W$ of $q_0$ in $N$,
we obtain $CR$-charts $(U;z_1,\hdots,z_{\nuup})$ and
$(W;w_2,\hdots,w_{\nuup})$ for new $CR$ structures which agree to
infinite order with the original ones at $p_0$ in $M$ and at
$q_0$ in $N$. The same map $\piup$ has a Lorentizian $CR$-noncharacteristic
singularity at $p_0$ also for the new locally $CR$-embeddable 
$CR$ structures, so that
the previous discussion applies. 
\end{proof}
\subsection{Example} Let $n\geq{1}$, $k\geq{1}$,  
$\nuup=n+k$, and
$N$ any $CR$ manifold of type $(n,k-1)$, contained in an
open neighborhood $G$ of $0$ in $\mathbb{C}^{\nuup-1}$, and minimal.
Let $w_1,\hdots,w_{\nuup-1}$ be the canonical holomorphic coordinates
of $\mathbb{C}^{\nuup-1}$ and assume that $dw_1$ and $d\bar{w}_1$
are linearly independent on $N$.\par
Let $\phiup:\mathbb{C}^{\nuup}\ni (z_1,\hdots,z_{\nuup})\to
(z_1,\hdots,z_{\nuup-1})\in\mathbb{C}^{\nuup-1}$ be the projection onto
the first $\nuup-1$ coordinates. If
\begin{equation*}
  M=\big\{z\in\mathbb{C}^{\nuup}\mid \phiup(z)\in{N},\;
\im{z}_1+{\sum}_{i=1}^{\nuup-1}z_i\bar{z}_i=z_\nuup\bar{z}_\nuup\big\},
\end{equation*}
then $M$ is a minimal $CR$ submanifold of type $(n,k)$ of
$\piup^{-1}(G)\subset\mathbb{C}^{\nuup}$, and 
the restriction of $\phiup$ describes a $CR$ map
$\piup:M\to{N}$ which has at $0$ a Lorentzian $CR$-noncharacteristic 
singularity. \par
In particular, there are $CR$ structures on a minimal Lorentzian
$CR$ manifol $M$ of arbitrary $CR$ codimension such that
a point $p_0\in{M}$ is critical for all $CR$ functions defined
on a neighborhood of $p_0$. 
\section{Tangential Cauchy-Riemann complexes 
and a global example}
\label{sec:4}
Throughout this section,
$M$ is a smooth $CR$ manifold, of positive $CR$ dimension $n$,
and arbitrary $CR$ codimension $k\geq{0}$.  
We set $\nuup=n+k$,
and denote by ${\Zt}$ the distribution of smooth complex vector fields
of type $(0,1)$ on $M$. \par
We recall that $\Ot_M(U)$ is the space of $CR$ functions
of class $\mathcal{C}^1$ on $U^{\text{open}}\subset{M}$. We set
$\Ot_M^{\infty}(U)=\Ot_M(U)\cap\mathcal{C}^{\infty}(U)$, and
$\Ot_M^{\infty}(\bar{U})=\Ot_M(U)\cap\mathcal{C}^{\infty}(\bar{U})$
for the restrictions to $\bar{U}$ of smooth functions on $M$, which are
$CR$ in $U$.
Likewise, when $\Omega$ is an open subset of a complex manifold $\mathbb{X}$,
we write ${\Ot}(\bar{\Omega})$ for
$\mathcal{C}^{\infty}(\bar{\Omega})\cap{\Ot}(\Omega)$.
\subsection{Definition of the $\bar\partial_M$-complexes}
Let $\Jtm$ be the ideal sheaf, corresponding to the characteristic
distribution $\Zt$ of $(0,1)$-vector fields on $M$.
Formal integrability of $\Zt$ is equivalent 
(see Lemma\,\ref{lem52}) to
\begin{equation}
  \label{eq:212}
  d\Jtm\subset\Jtm.
\end{equation}
This implies that also $d(\Jtm)^a\subset(\Jtm)^{a}$, where,
for each positive integer $a$,  $(\Jtm)^{a}$ is the
$a$-th exterior power of the ideal $\Jtm$, and we set
$(\Jtm)^0=\varOmega^*_M$. 
We can define cochain
complexes on $M$ by considering the quotients
$\mathscr{Q}^{a,*}_M=(\Jtm)^{a}/(\Jtm)^{a+1}$ and the map
$\bar{\partial}_M:\mathscr{Q}^{a,*}_M\to\mathscr{Q}^{a,*}_M$ induced 
on the quotients by
the exterior differential. We have \begin{equation*}
\mathscr{Q}^{a,*}_M={\bigoplus}_{q\geq{0}}\mathscr{Q}^{a,q}_M,\quad\text{
with 
$\mathscr{Q}^{a,q}_M=((\Jtm)^a\cap\varOmega_M^{a+q})/((\Jtm)^{a+1}
\cap\varOmega_M^{a+q})$},
\end{equation*}
and $\bar{\partial}_M(\mathscr{Q}^{a,q}_M)\subset\mathscr{Q}^{a,q+1}_M$.
This indeed was the intrinsic definition for the tangential Cauchy-Riemann
complexes on $CR$ manifolds given in \cite{N84}.\par
Let $\mathcal{Z}^{a,q}(U)=\{f\in\mathscr{Q}^{a,q}(U)\mid \bar{\partial}_Mf=0\}$ 
and $\mathcal{B}^{a,q}(U)=\bar{\partial}_M(\mathscr{Q}^{a,q-1}_M(U))$.
The quotient $\mathrm{H}^{a,q}(U)=\mathcal{Z}^{a,q}(U)/\mathcal{B}^{a,q}(U)$
is the 
\emph{cohomology group} of the smooth cohomology of $\bar{\partial}_M$ on $U$
in bidegree $(a,q)$. We set
\begin{equation*}
  \mathrm{H}^{a,q}(p_0)={\varinjlim}_{\; p_0\in{U}^{\text{open}}}\mathrm{H}^{a,q}(U)
\end{equation*}
for the group of germs of bidegree $(a,q)$-cohomology classes at $p_0$. 
\par\smallskip
Let us give a more explicit description of the equations involved in the
$\bar{\partial}_M$-complexes. 
\par
An element of $\mathcal{Z}^{a,q}(U)$ has
a representative $f\in\varOmega^{p+q}_M(U)$. The
conditions that $f\in(\Jtm)^a$ and that its class $[f]\in\mathcal{Z}^{a,q}(U)$ 
satisfies the 
integrability condition $\bar{\partial}_M[f]=0$
are expressed by
\begin{equation}\label{eq:213}
\begin{cases}
f\wedge\etaup_1\wedge\cdots\wedge\etaup_{\nuup-a+1}=0,\;
 \forall \etaup_1,\hdots,\etaup_{\nuup-a+1}\in\Jt_M^1(U),
 \\
 df\wedge\etaup_1\wedge\cdots\wedge\etaup_{\nuup-a+1}=0,\;
 \forall \etaup_1,\hdots,\etaup_{\nuup-a+1}\in\Jt_M^1(U),
 \end{cases}
\end{equation}
and the equation $\bar{\partial}_M\alpha=[f]$ for $\alpha\in\mathscr{Q}^{a,q-1}(U)$
is equivalent to finding $u\in\varOmega^{a+q-1}_M(U)$ such that 
\begin{equation}\label{eq:214}
 \begin{cases}
u\wedge\etaup_1\wedge\cdots\wedge\etaup_{\nuup-a+1}=0,\;
 \forall \etaup_1,\hdots,\etaup_{\nuup-a+1}\in\Jt_M^1(U),\\
 (du-f)\wedge\etaup_1\wedge\cdots\wedge\etaup_{\nuup-a+1}=0,\;
 \forall \etaup_1,\hdots,\etaup_{\nuup-a+1}\in\Jt_M^1(U).
 \end{cases}
\end{equation}
Both equation \eqref{eq:213} and \eqref{eq:214} are meaningful when 
$f,u$ are currents, and therefore we can consider the 
$\bar{\partial}_M$-complexes on currents, or require different degrees of
regularity on the data and the solution.
\subsection{Absence of Poincar\'e lemma}\label{sec:62}
In general, the $\bar\partial_M$-complexes are not acyclic (see e.g.
\cite{AFN81,AH72,HN06, N84}). In fact, the perturbations we used in the
previous section to deduce non complete-integrability results utilize 
elements of $\mathrm{H}^{0,1}(p_0)$. 
The arguments of \S\ref{sec4} provide simpler proofs of
the absence of the Poincar\'e lemma in some special cases.
We have e.g. (see \S\ref{sec4})
\begin{prop}\label{prop61}
Let $M$ be a $CR$ manifold of type $(n,1)$, which is locally $CR$-embeddable 
and Lorentzian at $p_0$, and let $(U;z_1,\hdots,z_{\nuup})$,
with $\nuup=n+1$,  be a $CR$ chart
centered at $p_0$ for which \eqref{eq:41a} holds and
$2z_1\bar{z}_1\geq\im{z}_{\nuup}$  on $U$. \par
 Let $\psi$ be a smooth function of one complex variable $\tauup$,
with compact support contained in 
$\{\im\tauup\geq{0}\}$. Then 
$\omegaup=z_1^{-1}\psi(z_{\nuup})d\bar{z}_{\nuup}$, continued by $0$ where
$z_1=0$, defines a smooth $\bar{\partial}_M$-closed $1$-form. A
necessary and sufficient condition for $\omegaup$ to be cohomologous
to $0$ is that
\begin{equation}\label{mom}
  \iint \tauup^h\psi(\tau)d\tau\wedge{d}\bar\tau=0,\;\forall h\in\mathbb{Z},\;
h\geq{0}.
\end{equation}
\end{prop}
\begin{proof}
Assume indeed that there is $u\in\mathcal{C}^1(U)$ such that 
$\bar\partial_Mu=[\omegaup]$. We keep the notation of \S\ref{sec4}
for $\piup,\;\omega,\; B$, and use integration on the fiber to define
\begin{equation*}
g(w)=\tfrac{-1}{2\pi i}\oint_{M_w}u\,dz_1,\quad\text{i.e.}\quad 
  \int_{\omega}g\phiup=\int_Uu\wedge 
dz_1\wedge\piup^*\phiup,\quad
\forall \phiup\in\varOmega^{2n}.
\end{equation*}
Clearly $g=0$ on $\partial\omega\cap{B}$ and furthermore we obtain
\begin{align*}
  \int_{\omega}dg\wedge\etaup=
\tfrac{-1}{2\pi{i}}\int_Udu\wedge{dz_1}\wedge\piup^*\etaup
\qquad\qquad\qquad\qquad\qquad
\\
=
\tfrac{-1}{2\pi{i}}\int z_1^{-1}\psi(z_{\nuup})d\bar{z}_{\nuup}\wedge{dz}_1\wedge
\piup^*\etaup
=\int_{\omega}\psi(w_n)d\bar{w}_n\wedge\etaup\qquad\\
\forall \etaup\in\varOmega^{n-1,n}_0(\omega).
\end{align*}
Thus $g$ satisfies
\begin{equation*}
  \begin{cases}
    \bar\partial{g}=\psi(w_n)d\bar{w}_n,&\text{on $\partial\omega$},\\
g=0, &\text{on $\partial\omega\cap{B}$}.
  \end{cases}
\end{equation*}
These equations imply that $\tauup\to g(0,\tauup)$ has compact support 
and hence that the equation $\partial{v}/\partial\bar{\tauup}=\psi(\tauup)$
has a solution with compact support. This is equivalent to the momentum
conditions \eqref{mom}
in the statement.
\end{proof}
Wy can repeat the same argument to show that the $\bar\partial_M$ complexes
are not acyclic in dimension $q$ when $M$ is 
strictly $q$-pseudoconvex at $p_0$.
Still restraining to type $(n,1)$ this condition means that, for
a suitable
$CR$ chart $(U;z_1,\hdots,z_{\nuup})$ centered at $p_0$ we have
\begin{equation}\label{eq:qx}
  \im{z}_{\nuup}+{\sum}_{i=q+1}^{\nuup}z_i\bar{z}_i={\sum}_{i=1}^qz_i\bar{z}_i+
O(|z|^3)\quad\text{on $U$}. 
\end{equation}
Here $\nuup=n+1$. By taking $U$ small we can assume that
$\im{z}_{\nuup}\leq 2{\sum}_{i=1}^qz_i\bar{z}_i$ on $U$. 
Set $\ell=\nuup-q$. By taking $U$ sufficiently small
we obtain a proper map
\begin{equation}\label{eq:pq}
  \piup:U\ni p\to (z_{q+1}(p),\hdots,z_{\nuup}(p))\in\omega\subset\mathbb{C}^{\ell},
\end{equation}
with $\omega$ the complement of a strictly convex open subset in an open
ball $B$ of $\mathbb{C}^{\ell}$, and $M_{w}\sim{S}^{2q-1}$ for $w\in\ring{\omega}$
and $M_{w}\sim{\text{a point}}$ for $w\in\partial\omega\cap{B}$. 
Let 
\begin{equation*}K_{q-1}(z_1,\hdots,z_q)
=\dfrac{{\sum}_{i=1}^q\bar{z}_i\,
\omegaup_i(z_1,\hdots,z_q)}{({\sum_{i=1}^qz_i\bar{z}_i})^q}\;\;\text{with}
\;\;
\omegaup_i(z_1,\hdots,z_q)=\dfrac{\partial}{\partial\bar{z}_i}\rfloor 
d\bar{z}_1\wedge\cdots{d\bar{z}_q}.
\end{equation*}
\begin{prop}
Let $M$ be a $CR$ manifold of type $(n,1)$, which is locally $CR$-embeddable 
and strictly $q$-pseudoconvex at $p_0$. Take
a $CR$-chart $(U,z_1,\hdots,z_{\nuup})$ with \eqref{eq:qx},
$\im{z}_{\nuup}\leq{2}z_1\bar{z}_1$ on $U$, and let
$\piup:M\to\omega$ be given by \eqref{eq:pq}.\par
If $\psi$ is a smooth function with compact support 
of one complex variable $\tauup$, with
$\supp\psi\subset\{\im\tauup\geq{0}\}$,  then 
$\omegaup=\psi(z_{\nuup})d\bar{z}_{\nuup}\wedge{K}_{q-1}(z_1,\hdots,z_q)$, 
continued by $0$ where
$z_1=0$, defines a smooth $\bar{\partial}_M$-closed $q$-form on $U$,
and \eqref{mom} is a 
necessary and sufficient condition for $\omegaup$ to be cohomologous
to $0$.
\end{prop}
\begin{proof}
The proof is analogous to that of Proposition\,\ref{prop61}.
Here we need to utilize, instead of Cauchy's formula, 
the identity
\begin{align}\label{eq:kop}
  \int_{\piup^{-1}(w)}K_{q-1}(z_1,\hdots,z_q)
\wedge{d}z_1\wedge\cdots\wedge{d}z_q=\tfrac{(-1)^{q(q-1)/2}(q-1)!}{(2\pi i)^q}
,\qquad \qquad\\ \notag
\forall w\in
\ring{\omega}\cap\supp\psi(w_{\ell}).
\end{align}
We have indeed
$d(K_{q-1}(z_1,\hdots,z_q)
\wedge{d}z_1\wedge\cdots\wedge{d}z_q)=0$ and then the value of the 
left hand side of \eqref{eq:kop} 
is a homology invariant. For all $w\in\ring{\omega}$, the fiber
$\piup^{-1}(w)$ is equivalent to the sphere 
$S^{2q-1}$. Then the value in  \eqref{eq:kop} can be computed by
integrating on $S^{2q-1}$ (see \cite{K67}).
\end{proof}
\subsection{Strictly pseudoconvex subdomains  of $CR$ manifolds}
Let $\Omega$ be an open set in $M$. 
Saying that 
$\partial\Omega$ is smooth at a point $p_0\in\partial\Omega$ means that
there is an open neighborhood $U$ of $p_0$ in $M$ and a smooth
real valued function $\phi\in\mathcal{C}^{\infty}(U,\mathbb{R})$
such that
\begin{equation}
  \label{eq:3.1}
  \Omega\cap{U}=\{p\in{U}\mid \phi(p)<0\},\quad d\phi(p_0)\neq{0}.
\end{equation}
\begin{defn}
We say that $\Omega$ is \emph{strictly pseudoconvex} at $p_0$ if
$\mathfrak{L}_{d\phi(p_0)}>0$ on $\mathbf{Z}_{p_0}M\cap\ker{d}\phi(p_0)$.
\end{defn}
Note that $\mathbf{Z}_{p_0}M\cap\ker{d}\phi(p_0)=\mathbf{Z}_{p_0}M$
when $d\phi(p_0)\in{H}^0_{p_0}M$.\par
We need to introduce a weaker notion of $CR$ function on $M$. 
\begin{defn}
If 
$U^{\text{open}}\subset{M}$, we say that $u\in\mathcal{C}^0(U)$ 
is $CR$ if
\begin{equation*}
  \int_U u\,d\eta=0,\quad\forall 
\eta\in(\Jtm)^{\nuup}(U)\cap\varOmega^{2n+k-1}_{M,0}(U),
\end{equation*}
where we indicate by $\varOmega^*_{M,0}(U)$ smooth exterior forms having 
compact support in~$U$.
\end{defn}
If $M$ is a $CR$ submanifold
of a complex manifold $\mathbb{X}$ and is minimal at $p_0$, then
all germs of continuous 
$CR$ functions extend holomorphically to wedges in $\mathbb{X}$ whose 
edges contain neighborhoods of $p_0$
(see \cite{Tu88, Tu90}). We set $\Ot_M^{\text{cont}}$ for the sheaf of
germs of continuous $CR$ functions on $M$. For every $U^{\text{open}}\subset{M}$,
the space
$\Ot_M^{\text{cont}}(U)$ is Fr\'echet for the topology of uniform
convergence on the compact subsets of~$U$.
\begin{defn}
Let $\Omega$ be an open subset of $M$,
$f\in{\Ot}_M(\Omega)$, and $p_0\in\partial\Omega$.
We say that $f$ 
weakly $CR$-extends beyond $p_0$ if there exists a connected
open neighborhood $U$ of $p_0$ in $M$ and 
$g\in{\Ot}_M^{\text{cont}}(U)$ such that $\{p\in\Omega\cap{U}
\mid g(p)=f(p)\}$ has a non empty interior.  
\end{defn}

\begin{prop}\label{prop:3.2}
  Let $M$ be a $CR$ submanifold of 
$CR$-dimension $n\geq{1}$ and arbitrary $CR$ codimension $k\geq{0}$ of
a Stein manifold $\mathbb{X}$, 
and $\Omega$ a smooth strictly pseudoconvex domain in $\mathbb{X}$,
with $\partial\Omega\cap{M}\neq\emptyset$.
Let $M_0$ be the set of points of $M\cap\partial\Omega$
at which $M$ is minimal.
Consider on ${\Ot}(\bar{\Omega})$
the natural Fr\'echet topology of uniform convergence with all derivatives
on the compact subsets of $\bar\Omega$. 
\par
Then the set
$\mathbb{E}$ of the elements 
$f\in{\Ot}(\bar{\Omega})$ such that 
$f|_{M\cap\bar{\Omega}}$ weakly $CR$-extends beyond some point $p\in{M}_0$ 
is of the first Baire category in 
${\Ot}(\bar{\Omega})$.
\end{prop}
\begin{proof} 
Fix a point $p_0\in{M}_0$ and a 
countable fundamental system of open neighborhoods $\{V_\nu\}_{\nu\in\mathbb{N}}$
of $p_0$ in $\mathbb{X}$. 
We can assume that for every $\nu\in\mathbb{N}$ 
the intersection $M\cap{V}_{\nu}$ is connected and contained in the set
of minimal points of $M$. In particular, two $CR$ functions on $M\cap{V}_{\nu}$
which agree on some non empty open subset of $M\cap{V}_{\nu}$, are equal on
all $M\cap{V}_{\nu}$.
For each $\nu$ the set 
\begin{equation*}
 \mathbb{F}_{\nu}=\{(f,g)\in{\Ot}(\bar{\Omega})
\times{\Ot}_M^{\text{cont}}(M\cap
 V_{\nu})\mid g=f\;\text{on}\; M\cap\Omega\cap{V}_{\nu}\}
\end{equation*}
is a Fr\'echet space, being a closed subspace of 
${\Ot}(\bar{\Omega})\times
{\Ot}_M^{\text{cont}}(M\cap V_{\nu})$.\par
Let $\pi_{\nu}:\mathbb{F}_{\nu}\to{\Ot}(\bar\Omega)$ be the
projection into the first component.\par
Assume by contradiction that the set of $f\in{\Ot}(\bar\Omega)$ 
such that $f|_{M\cap\Omega}$ can
be weakly $CR$-continued beyond $p_0$ 
is of the second Baire category.  Then 
$ \bigcup_{\nu\in\mathbb{N}}\pi_{\nu}(\mathbb{F}_{\nu})$
is of the second Baire category, and hence 
at least one $\pi_{\nu_0}(\mathbb{F}_{\nu_0})$ is of the second
Baire category in ${\Ot}(\bar\Omega)$.
Hence
$\pi_{\nu_0}(\mathbb{F}_{\nu_0})={\Ot}(\bar\Omega)$ and,
by Banach-Shauder's theorem, the map 
$\pi_{\nu_0}:\mathbf{F}_{\nu_0}\to{\Ot}(\bar\Omega)$ is open
(see e.g \cite[\S{2.1}]{SW99}). 
By the assumption of minimality and 
the fact that
Lemma \ref{lem:3.1} also applies to continuous $CR$ functions, for each
$f\in{\Ot}(\bar\Omega)$ there is at most one 
$g\in{\Ot}_M^{\text{cont}}(M\cap{V_{\nu}})$ such that
$(f,g)\in\mathbb{F}_{\nu}$. \par
Thus, we conclude that 
there is a relatively compact open neighborhood $V$ of
$p_0$ in $\mathbb{X}$ such that  
\begin{gather}
 \forall f\in{\Ot}(\bar\Omega)\;\;\exists ! g\in{\Ot}_M(
 M\cap{V}) \quad\text{with $f=g$ on $M\cap\Omega\cap{V}$ and}\\
 \label{eq:3.3}
 |g(p)|\leq \|f\|_{\ell,K}, \quad \forall p\in{M}\cap{V},
\end{gather}
where $K$ is a compact subset of $\bar\Omega$ and $\|f\|_{\ell,K}$ a 
seminorm involving the derivatives of $f$ up to order $\ell$ on $K$.
Let $K_{\nu}$ be a sequence of compact subsets of $\mathbb{X}$
such that $K_{\nu+1}\Subset\text{int}({K}_{\nu})$, 
for all $\nu\in\mathbb{N}$, and
$\bigcap_{\nu}K_{\nu}=K$. By Cauchy's inequalities, there are constants
$C_{\nu}>0$ such that
\begin{equation*}
 \|f\|_{\ell,K}\leq C_{\nu}{\sup}_{K_{\nu}}|f|,\quad\forall f\in{\Ot}
 (\mathbb{X}).
\end{equation*}
Using \eqref{eq:3.3} we obtain 
\begin{equation*}
 |f(p)|\leq C_{\nu}{\sup}_{K_{\nu}}|f|,\quad \forall f\in{\Ot}(\mathbb{X}),\;\;
 \forall p\in{M}\cap{V}.
\end{equation*}
By applying this inequality to the positive integral powers of 
the entire functions on $\mathbb{X}$, we obtain that
\begin{gather*}
 |f(p)|\leq 
 C_{\nu}^{{1/h}}\;
 {\sup}_{K_{\nu}}|f|,\quad \forall f\in{\Ot}(\mathbb{X}),\;\;\forall
 p\in{M}\cap{V},\;\;
 \forall 0<h\in\mathbb{N}\\
  \Longrightarrow 
  |f(p)|\leq 
  {\sup}_{K_{\nu}}|f|,\quad \forall f\in{\Ot}(\mathbb{X}),\;\;
  \forall p\in{M}\cap{V}.
\end{gather*}
Since ${\bigcap}_{\nu}K_{\nu}=K$, we obtain 
\begin{equation*}
 |f(p)|\leq{\sup}_{K}|f|,\quad\forall f\in{\Ot}(\mathbb{X}).
\end{equation*}
But this gives a contradiction, because the fact that
$\Omega$ is strictly pseudoconvex implies that $\bar{\Omega}$
is holomorphically convex in $\mathbb{X}$. \par
We showed that,
for every point $p\in{M}_0$,  
the set $\mathbb{E}_{p}$ of
$f\in{\Ot}(\bar\Omega)$ 
for which $f|_{M\cap\Omega}$ weakly $CR$-extends beyond $p$
is of the first Baire category
in ${\Ot}(\bar\Omega)$. The set $M_0$ is separable. Then
we take a dense sequence $\{p_{\nu}\}$ in $M_0$ and we observe
that $\mathbb{E}=\bigcup_{\nu\in\mathbb{N}}\mathbb{E}_{p_{\nu}}$
is of the first Baire category,
being a countable union of sets of the first Baire category.
\end{proof}

\subsection{The Cauchy problem for $\bar{\partial}_M$}
Let $F^{\text{closed}},U^{\text{open}}\subset{M}$. Set, for all $a=0,\hdots,\nuup$,
$q=0,\hdots,n$, 
\begin{displaymath}
\mathscr{Q}^{a,q}(U;F)=\{f\in\mathscr{Q}^{a,q}(U)\mid f=0^{\infty}\;\text{on}\;
F\cap{U}\}. 
\end{displaymath}
Clearly $(\mathscr{Q}^{a,*}(U;F),\bar{\partial}_M)$ is a subcomplex
of $(\mathscr{Q}^{a,*}(U),\bar{\partial}_M)$
and we can consider the cohomology groups
\begin{equation*}
  \mathrm{H}^{a,q}(U,F)=\ker(\bar{\partial}_M:\mathscr{Q}^{a,q}(U;F)\to
\mathscr{Q}^{a,q+1}(U;F))/\bar\partial_M\mathscr{Q}^{a,q-1}(U;F)
\end{equation*}
and also their germs
\begin{equation*}
   \mathrm{H}^{a,q}(p_0,F)={\varinjlim}_{\; p_0\in{U}^{\text{open}}}\mathrm{H}^{a,q}(U,F).
\end{equation*}

We obtain from Proposition \ref{prop:3.2}
\begin{prop}
  \label{prop:3.3}
Let $M$ be a $CR$ submanifold of positive $CR$ dimension $n$ of a
Stein manifold $\mathbb{X}$. 
Let $M'$ be the set of points of $M$ where $M$ is minimal.
Let $\Omega$ be a strictly pseudoconvex 
open subset of $\mathbb{X}$, with 
$M_0=M'\cap\partial\Omega\neq\emptyset$.
Then the set of $\alpha\in
\mathrm{H}^{0,1}(M,\bar{\Omega}\cap{M})$ 
that restrict to the zero class of
$\mathrm{H}^{0,1}(p,\bar{\Omega}\cap{M})$ for some
$p\in{M}_0$
is infinite dimensional.\par
More precisely, if we consider the Fr{\'e}chet space
\begin{equation*}
  \mathscr{F}=\{f\in\varOmega^1({M})\mid f|_{\bar{\Omega}\cap{M}}=0^{\infty},
\;\; df\in\Jtm\},
\end{equation*}
then the set of $f\in\mathscr{F}$ for which we can find $p\in{M}_0$ and
$u_{(p)}\in{\mathcal{C}}^0_{(p)}$ with $u_{(p)}$ vanishing on
$\bar{\Omega}\cap{M}$, and 
$\bar{\partial}_Mu_{(p)}=f_{(p)}$, span a linear subspace of infinite
dimensional codimension in 
$\mathscr{F}$. Here we wrote $f_{(p)}$ for the germ of $f$ at $p$, and 
the equality $\bar{\partial}_Mu_{(p)}=f_{(p)}$
has to be interpreted in the weak sense: it means that there is 
an open neighborhood $U_p$ of $p$ and a continuous function 
$u\in\mathcal{C}(U_p)$, vanishing on $\bar{\Omega}\cap{U}_p$, such that
\begin{equation*}
  \int_M u\, d\etaup = -\int_M f\wedge \etaup,\quad\forall
\etaup\in (\Jtm)^{\nuup}(U_p)\cap\varOmega^{2n+k-1}_{0}(U_p). 
\end{equation*}
\end{prop}
\begin{proof}
  Let $g\in{\Ot}(\bar\Omega)$. By Whitney's extension theorem,
there is a smooth complex valued function 
$\tilde{g}$ on $\mathbb{X}$ with
$\tilde{g}=g$ on $\bar{\Omega}$. Then $f=\bar\partial\tilde{g}|_M$ 
is an element of $\mathscr{F}$. Let $w=\tilde{g}|_{M}$.
Then $w$ is a smooth function on $M$ and its restriction to $M\cap\Omega$
is $CR$.
If $p\in{M}_0$ and 
$u_{(p)}\in{{}}{\mathcal{C}}^0_{(p)}$ solves $\bar{\partial}_Mu_{(p)}=
f_{(p)}$,  then $w_{(p)}+u_{(p)}$ yields a 
weak $CR$ extension
of $w|_{M\cap\Omega}$ beyond $p$. Therefore 
the thesis follows by Proposition\,\ref{prop:3.2}.
\end{proof}
Let $N$ be a smooth submanifold of $M$. 
Its conormal bundle in $M$ at $p\in{N}$ is
\begin{equation*}
  T^*_{N,p}M=\{\xi\in{T}^*_pM\mid \xi(v)=0,\;\forall v\in{T}_pN\}.
\end{equation*}
\begin{dfn} We say that $N$ is \emph{characteristic} at $p\in{N}$ if
$T^*_{N,p}M\cap{H}^0_pM\neq\{0\}$.
\end{dfn}
Equivalently, $N$ is non characteristic at $p\in{N}$ if it contains a germ
$(N',p)$ of $CR$ submanifold of type $(0,\nuup)$ of $M$.
We have the following
\begin{lem} \label{lem:3.4}
Let $M$ be a $CR$ manifold of positive $CR$ dimension $n$, and
$N$ a smooth submanifold of $M$. Let $U$ be an open neighborhood
of a non characteristic point $p_0$ of ${N}$.
\par
If $f\in\mathscr{Q}^{0,1}(U,N)$ and $u\in\mathcal{C}^{\infty}(U)$
solves
\begin{equation}
  \label{eq:3.7}\begin{cases}
  \bar{\partial}_Mu=f &\text{on}\;\;U,\\
  u=0 &\text{on}\;\; N
\end{cases}
\end{equation}
then $u$ vanishes to infinite order at $p_0$.
\end{lem}
\begin{proof}
The statement follows by the uniqueness in the formal
non characteristic Cauchy problem. 
\end{proof}
\subsection{The canonical bundle} 
The sheaf $\mathscr{Q}^{\nuup,0}=
(\Jtm)^\nuup\cap{\varOmega}^{\nuup}_M$
is the sheaf of germs of smooth sections of a line bundle $KM$ on $M$,
that is called the \emph{canonical bundle} of $M$.
\par
Let $U^{\text{open}}\subset{M}$, and assume that
$\thetaup_1,\hdots,\thetaup_{\nuup}\in\Jtm(U)$ give a basis of
$\mathbf{Z}^0_pM$ at all points $p\in{U}$. Since
$d\thetaup_j\in\Jtm(U)$, we obtain
\begin{equation*}
  d(\thetaup_1\wedge\cdots\wedge\thetaup_{\nuup})=\alphaup\wedge
\thetaup_1\wedge\cdots\wedge\thetaup_{\nuup},\;\;\text{with}\;\;
\alphaup\in\varOmega^1(U).
\end{equation*}
The form $\alphaup$ is uniquely determined modulo $\Jtm$, and thus
defines a unique element $[\alphaup]\in\mathscr{Q}^{0,1}(U)$. 
By differentiating we get
\begin{align*}
  0&=d^2(\thetaup_1\wedge\cdots\wedge\thetaup_{\nuup})=
(d\alpha)\wedge \thetaup_1\wedge\cdots\wedge\thetaup_{\nuup}-
\alpha\wedge{d}(\thetaup_1\wedge\cdots\wedge\thetaup_{\nuup})\\
&=(d\alpha)\wedge \thetaup_1\wedge\cdots\wedge\thetaup_{\nuup}
-\alpha\wedge\alphaup\wedge
\thetaup_1\wedge\cdots\wedge\thetaup_{\nuup}\\
&=(d\alpha)\wedge \thetaup_1\wedge\cdots\wedge\thetaup_{\nuup}
\;\Longleftrightarrow \bar{\partial}_M[\alpha]=0.
\end{align*}
If $\omegaup$ is another non zero 
section of $KM$, defined in a neighborhood
of a point $p_0\in{U}$, we have 
$\omegaup=e^u\thetaup_1\wedge\cdots\wedge\thetaup_{\nuup}$
on a neighborhood $U_{p_0}$ of $p_0$ in $U$ and hence
\begin{equation*}
  d\omegaup = e^u(du+\alphaup)\wedge\thetaup_1\wedge\cdots\wedge\thetaup_{\nuup}
\;\;\text{on $U_{p_0}$.}
\end{equation*}
Thus we obtain
\begin{lem}
There is a section $\psiup\in\Gamma(M,{{}}{H}^{0,1})$ of the sheaf
of germs of cohomology classes of bidegree $(0,1)$ such that
\begin{equation}
  \label{eq:35} \begin{cases} 
\forall p\in{M},\;\;\forall \omegaup\in{{}}{\Gamma}_{(p)}(KM),\;\;
\text{with $\omegaup(p)\neq{0}$},\;\;
\exists \alpha\in{{}}{\varOmega}^1_{(p)}
\;\;\text{such that}\\
d\omegaup(p)=\alpha(p)\wedge\omegaup(p).\;\;
 \bar\partial_M[\alpha]=0,\;\; \llbracket\alpha
\rrbracket\in\psiup(p).
\end{cases}
\end{equation}
Here $\llbracket\alpha
\rrbracket$ is the element of $\mathrm{H}^{0,1}(p)$ defined by
$[\alpha]\in{{}}{\mathcal{Z}}^{0,1}_{(p)}$.
\end{lem}
\begin{lem}
If $M$ is locally $CR$-embeddable at $p$, then $\psiup(p)=0$. \par  
If $\Zt$ is completely integrable at $p$, then, for 
$\alpha\in{{}}{\varOmega}^1_{(p)}$ satisfying \eqref{eq:35}
there is $u\in{{}}{\mathcal{C}}^0_{(p)}$ such that
$\bar\partial_Mu=[\alpha]$. 
\end{lem}
\subsection{$CR$-foldings} Let us recall the notion of a fold
singularity for a smooth map (see e.g. \cite{E1970}).
Let $M$, $N$ be real smooth manifolds of the same real dimension. 
A map $\piup:M\to{N}$ is a 
\emph{fold-map}
if there is a smooth submanifold $V$ of $M$ such that:
\begin{list}{-}{}
\item the restriction of $\piup$ to $M\setminus{V}$ is a two-sheeted covering;
\item the restriction of $\piup$ to ${V}$ is a smooth immersion;
\item there is an involution  $\sigmaup:U\to{U}$ of a 
tubular neighborhood $U$ of $V$ in $M$ such that 
$\piup\circ\sigmaup=\piup$ on $U$.
\end{list}
The corresponding notion in $CR$ geometry will be 
a folding about a $CR$-divisor.
Let us introduce the notions that we will utilize in the sequel.
\begin{defn}[Smooth $CR$-divisors] A \emph{smooth $CR$-divisor} 
of a $CR$ manifold $M$
is a smooth submanifold $D$ of $M$, having real codimenison $2$, 
such that
for each $p_0\in{D}$ there is an open neighborhood $U_{p_0}$ 
of $p_0$ in $M$ and a
function $f\in\Ot^{\infty}_M(U_{p_0})$ such that
$D\cap{U}_{p_0}=\{p\in {U}_{p_0}\mid f(p)=0\}$ and 
$df(p_0)\wedge{d}\bar{f}(p_0)\neq{0}$.
\end{defn}
\begin{defn}[$CR$-folding] Let $M,N$ be $CR$ manifolds having the same 
$CR$ dimension. A $CR$-folding is a proper 
smooth $CR$ map $\piup:M\to{N}$ such that
there exists a smooth $CR$-divisor $D$ on $M$ with the properties:
\begin{list}{-}{}
\item the restriction of $\piup$ to $M\setminus{D}$ is 
a $CR$-submersion and a smooth circle bundle;
\item 
the restriction of $\piup$ to $D$ is a 
 smooth $CR$ immersion;
\item there is a tubular neighborhood $U$ of $D$ in $M$ and 
an $\mathbf{S}^1$-action on the fibers of $\piup|_{U}:U\to\piup(U)$,
for which $D$ is the set of fixed points.
\end{list}
\end{defn}
\begin{exam} The map
$S^3=\{(z,w)\mid z\bar{z}+w\bar{w}=1\}\ni(z,w)
\xrightarrow{\;\piup\;}{z}\in\mathbb{C}$ is a $CR$-folding. 
The set $D=\{(z,0)\mid z\bar{z}=1\}=\{(z,w)\in{S}^3\mid w=0\}$ 
is a smooth
$CR$-divisor in $S^3$, and $\piup^{-1}(z_0)=\{(z_0,w)
\mid w\bar{w}=1-z_0\bar{z}_0\}$.
We can define, in the
tubular neighborhood $U=S^3\cap\{|z|>0\}$, an 
$\mathbf{S}^1$-action 
by $e^{i\theta}\cdot (z,w)=
(z,e^{i\theta}w)$.
\end{exam}
\begin{exam} \label{ex45} Let $Q\subset\mathbb{CP}^{\nuup}$,
with $\nuup\geq{3}$, be
the ruled real projective  quadric, which is a $CR$ submanifold of
type $(n,1)$, with $n=\nuup-1$, and 
Levi signature $(1,n-1)$. For a suitable choice of homogeneous
coordinates, $Q$ has equation
\begin{equation}
  \label{eq:5.1}
  z_0\bar{z}_0+z_1\bar{z}_1={\sum}_{j=2}^{\nuup}z_j\bar{z}_j.
\end{equation}
The point $p_0\equiv(1,0,\hdots,0)$ does
not belong to $Q$. Hence, by 
associating to each $p\in{Q}$ the complex line
$p_0p$, we obtain a map $\pi:Q\to\mathbb{CP}^{n}$,
where $\mathbb{CP}^{n}$ is the set of complex lines through $p_0$
of $\mathbb{CP}^{n+1}$. After identifying 
$\mathbb{CP}^{n}$ with the
hyperplane $\{z_0=0\}$, the map $\pi$ is described in homogeneous coordinates by
\begin{equation}
  \label{eq:5.2}
  \pi:(z_0,z_1,\hdots,z_\nuup)\longrightarrow (z_1,\hdots,z_\nuup).
\end{equation}
This map is a $CR$-folding of $M$ into $\mathbf{CP}^n$, 
with divisor $Q\cap\{z_0=0\}$, 
whose image
is the complement of
an open  ball in the projective space:
\begin{equation}
  \label{eq:5.3}
  \pi(Q)=\big\{z_1\bar{z}_1\leq {\sum}_{j=2}^{\nuup}z_j\bar{z}_j\big\}.
\end{equation}
\par 
To describe the generators of the ideal sheaf 
$\Jt_Q$ of $Q$, we consider the covering $\{U,V\}$ of $Q$
with $U=\{z_0\neq{0}\}$, $V=\{z_1\neq{0}\}$. Then $\Jt_Q$ 
is defined by the generators
\begin{gather*}
  z_0^{-1}dz_1,\;z_0^{-1}dz_2,\hdots,\;z_0^{-1}dz_\nuup\quad\text{on}\; U\cap{Q},\\
z_1^{-1}dz_0,\; z_1^{-1}dz_2,\hdots,\;z_1^{-1}dz_\nuup\quad\text{on}\; V\cap{Q}.
\end{gather*}
The canonical bundle is generated by $z_0^{-\nuup}dz_1\wedge
dz_2\wedge\cdots{dz_\nuup}$
on $U\cap{Q}$ and by $z_1^{-\nuup}dz_0\wedge
dz_2\wedge\cdots{dz_\nuup}$ on $V\cap{Q}$. 
\end{exam}

\subsection{Construction of a locally non $CR$-embeddable
perturbation}\label{s:5}
In this subsection we describe a procedure to define 
a non locally $CR$-embeddable $CR$ structure on a neighborhood of
$p_0$ in $M$ that we will use in \S\ref{s6} to produce a
\textit{global} example.\par
Let $M$, $N$ be $CR$ manifolds of type $(n,k)$, $(n,k-1)$ respectively.
Assume that there is a $CR$ map $\piup:M\to{N}$ having a Lorentzian
singularity and a $CR$-folding 
at $p_0$, and that $M$, $N$ are locally
$CR$-embeddable at $p_0$, $q_0$, respectively. 
We are in fact in the situation of Lemma\,\ref{lem516} and we keep the
notation therein. 
\par
Let $\alphaup$ be a $(0,1)$-form on $\omega$,
with $\bar{\partial}_N\alphaup=0$ and $\alphaup=0$ on
$\omega_+$. Then
 $z_1^{-1}\piup^*\alphaup$ is well defined and smooth because
$\piup^*\alphaup$ vanishes to infinite order on $D=\{z_1=0\}$, and
we may consider on $U$ the $CR$-structure
with ideal sheaf $\Jt'_U$ generated by
\begin{equation}\label{eq:62}
  dz_1-z_1^{-1}\piup^*\alphaup, \; dz_2,\; \hdots,\; dz_{\nuup}.
\end{equation}
This new $CR$ structure agrees with the original one
to infinite order at all points of~$D$. 
If this new $CR$ structure admits a $CR$-chart centered at $p_0$,
then there is a smooth function $u$, defined on a neighborhood
of $p_0$, with $u(p_0)=0$ and 
\begin{equation}\label{eq63}
  d(e^u(dz_1-z_1^{-1}\piup^*\alphaup)\wedge dz_2\wedge \cdots\wedge
dz_{\nuup})=0.
\end{equation}
Then we obtain
\begin{equation*}
  du\wedge (dz_1-z_1^{-1}\piup^*\alphaup)\wedge dz_2\wedge \cdots\wedge
dz_{\nuup}+z_1^{-2}dz_1\wedge\piup^*\alphaup\wedge dz_2\wedge \cdots\wedge
dz_{\nuup}=0,
\end{equation*}
from which we get
\begin{equation*}
 (d(z_1^2u)-\piup^*\alphaup^*) \wedge\frac{dz_1}{z_1}\wedge
 dz_2\wedge \cdots\wedge
dz_{\nuup} +d(u\,\piup^*\omega^*\wedge
 dz_2\wedge \cdots\wedge
dz_{\nuup})=0.
\end{equation*}
Next we integrate on the fiber. For $q\in{W}\cap\omega_-$,
where $W$ is a suitable small neighborhood of $q_0$ in $\omega$, 
 we obtain
\begin{equation*}
  w(q)=\piup_*(z_1^2u)(q)=\tfrac{1}{2\pi{i}}\oint_{\piup^{-1}(q)}z_1u \,dz_1,
\end{equation*}
and therefore 
\begin{align*}
& dw(q)\wedge
dz_2\wedge \cdots\wedge
dz_{\nuup} =\tfrac{1}{2\pi i}\oint_{\piup^{-1}(q)}d(z_1^2u)\frac{dz_1}{z_1}\wedge
dz_2\wedge \cdots\wedge
dz_{\nuup} \\
&=\tfrac{1}{2\pi i} \oint_{\piup^{-1}(q)} \piup^*\alphaup^*\frac{dz_1}{z_1}\wedge
dz_2\wedge \cdots\wedge
dz_{\nuup}  
=\alphaup\wedge
dz_2\wedge \cdots\wedge
dz_{\nuup},
\end{align*}
because 
\begin{equation*}
  \oint_{\piup^{-1}(q)}d(u\,\piup^*\omega^*\wedge
 dz_2\wedge \cdots\wedge
dz_{\nuup})=0.
\end{equation*}
We observe that $w=0^2$ on $W\cap\partial\omega_-$ and that the equality
established above means that $w$ satisfies
\begin{equation*}
  \bar\partial_Nw=[\alphaup]\quad\text{on $N$}. 
\end{equation*}
By Lemma\,\ref{lem:3.4}, we actually have $w=0^{\infty}$ on $\partial{N}$. 
By Proposition\,\ref{prop:3.3} there are $\bar{\partial}_N$-closed forms
$\alphaup$, of type $(0,1)$, 
vanishing to infinite order on $W\cap\partial\omega_-$, for which the
Cauchy problem
\begin{equation*}
  \begin{cases}
    \bar{\partial}_Nw=\alphaup&\text{on $W\cap\omega_-$},\\
w=0^{\infty}&\text{on $W\cap\partial\omega_-$}
  \end{cases}
\end{equation*}
has no solution when $W$ is any open neighborhood of $q_0$,
hence yielding non $CR$-embeddable $CR$ structures on $U$ which agree to infinite
order with the original one.
\subsection{A global example}\label{s6}
In this section we construct a 
non locally $CR$-embeddable $CR$ structure on
the Lorentzian quadric of Example\,\ref{ex45}. \par
In $\mathbb{CP}^{\nuup}$, with 
homogeneous coordinates $z_0,z_1,\hdots,z_{\nuup}$,  
for $\nuup\geq{2}$, we consider the quadric
$Q=\{z_0\bar{z}_0+z_1\bar{z}_1 = z_2\bar{z}_2+\cdots+z_{\nuup}\bar{z}_{\nuup}\}$.  
Fix the divisor $D=\{z_0=0\}$ and the global $CR$-folding
$Q\to{N}$ where 
$N=\{z_1\bar{z}_1\leq z_2\bar{z}_2+\cdots+z_{\nuup}\bar{z}_{\nuup}\}
\subset\mathbb{CP}^{\nuup-1}$
is a strictly pseudoconcave closed domain in $\mathbb{CP}^{\nuup -1}$. 
The closure of its complement is  an Euclidean ball 
${B}$ of $\mathbb{C}^{\nuup-1}=\mathbb{CP}^{\nuup-1}\setminus\{z_1=0\}$.
Fix a smooth function $f$, with compact support in $\mathbb{C}^{\nuup -1}$,
which is holomorphic on $B$, but cannot be continued holomorphically beyond
any point of $\partial{B}$, and let $\alphaup$ be the restriction of
$\bar{\partial}f$ to $N$. We observe that $\zeta=\dfrac{z_1}{z_0}$ 
is meromorphic on $\mathbf{CP}^m$ and that $\zeta\piup^*\alphaup$
is well-defined on $Q$. We define a new $CR$ structure on $Q$
by the ideal sheaf having 
generators
\begin{gather*}
z_1^{-1}dz_0-\zeta\piup^*\alphaup^*,\; z_1^{-1}dz_2,\;\hdots,\; z_1^{-1}dz_{\nuup}
\quad\text{on $M\cap\{z_1\neq{0}\}$},\\
z_0^{-1}dz_1,\; z_0^{-1}dz_2,\;\hdots,\; z_0^{-1}dz_{\nuup}
\quad\text{on $M\setminus\supp\piup^*\alphaup$}.
\end{gather*}
Note that the hyperplane $\{z_1=0\}$ in $\mathbb{CP}^{\nuup -1}$ does not intersect
the support of $\alphaup$, and hence the ideal sheaf is well defined.
By the argument in \S\ref{s:5}, with this new $CR$ structure
$Q$ is not locally $CR$-embeddable at all points of the divisor $D$. 
Thus we have obtained
\begin{thm}
There are $CR$ structures of type $(m-1,1)$ on the Lorentzian quadric
$Q$ that are not locally $CR$-embeddable at all points of a hyperplane section 
of~$Q$. \qed
\end{thm}

\bibliographystyle{amsplain}
\renewcommand{\MR}[1]{}
\providecommand{\bysame}{\leavevmode\hbox to3em{\hrulefill}\thinspace}
\providecommand{\MR}{\relax\ifhmode\unskip\space\fi MR }
\providecommand{\MRhref}[2]{%
  \href{http://www.ams.org/mathscinet-getitem?mr=#1}{#2}
}
\providecommand{\href}[2]{#2}


\begin{thebibliography}{10}

\bibitem{Ak87}
Takao Akahori, \emph{A new approach to the local embedding theorem of
  {CR}-structures for {$n\geq 4$} (the local solvability for the operator
  {$\overline\partial_b$} in the abstract sense)}, Mem. Amer. Math. Soc.
  \textbf{67} (1987), no.~366, xvi+257. \MR{MR888499 (88i:32027)}

\bibitem{AHNP}
Andrea Altomani, C.~Denson Hill, Mauro Nacinovich, and Egmont Porten,
  \emph{Complex vector fields and hypoelliptic partial differential operators},
  Ann. Inst. Fourier (Grenoble) \textbf{60} (2010), no.~3, 987--1034.
  \MR{2680822}

\bibitem{AFN81}
Aldo Andreotti, Gregory Fredricks, and Mauro Nacinovich, \emph{On the absence
  of {P}oincar\'e lemma in tangential {C}auchy-{R}iemann complexes}, Ann.
  Scuola Norm. Sup. Pisa Cl. Sci. (4) \textbf{8} (1981), no.~3, 365--404.
  \MR{MR634855 (83e:32021)}

\bibitem{AH72}
Aldo Andreotti and C.~Denson Hill, \emph{E. {E}. {L}evi convexity and the
  {H}ans {L}ewy problem. {II}. {V}anishing theorems}, Ann. Scuola Norm. Sup.
  Pisa (3) \textbf{26} (1972), 747--806. \MR{MR0477150 (57 \#16693)}

\bibitem{BR87}
M.~S. Baouendi and L.~P. Rothschild, \emph{Embeddability of abstract {CR}
  structures and integrability of related systems}, Ann. Inst. Fourier
  (Grenoble) \textbf{37} (1987), no.~3, 131--141. \MR{MR916277 (89c:32053)}

\bibitem{BR90}
M.~S. Baouendi and Linda~Preiss Rothschild, \emph{Cauchy-{R}iemann functions on
  manifolds of higher codimension in complex space}, Invent. Math. \textbf{101}
  (1990), no.~1, 45--56. \MR{MR1055709 (91j:32020)}

\bibitem{BT81}
M.~S. Baouendi and F.~Tr{\`e}ves, \emph{A property of the functions and
  distributions annihilated by a locally integrable system of complex vector
  fields}, Ann. of Math. (2) \textbf{113} (1981), no.~2, 387--421. \MR{MR607899
  (82f:35057)}

\bibitem{BCH}
Shiferaw Berhanu, Paulo~D. Cordaro, and Jorge Hounie, \emph{An introduction to
  involutive structures}, New Mathematical Monographs, vol.~6, Cambridge
  University Press, Cambridge, 2008. \MR{2397326 (2009b:32048)}

\bibitem{BdM74}
L.~Boutet~de Monvel, \emph{Int\'egration des \'equations de {C}auchy-{R}iemann
  induites formelles}, S\'eminaire {G}oulaouic-{L}ions-{S}chwartz 1974--1975;
  \'{E}quations aux deriv\'ees partielles lin\'eaires et non lin\'eaires,
  Centre Math., \'Ecole Polytech., Paris, 1975, pp.~Exp. No. 9, 14.
  \MR{MR0409893 (53 \#13645)}

\bibitem{BHN02}
J.~Brinkschulte, C.~Denson Hill, and M.~Nacinovich, \emph{Obstructions to
  generic embeddings}, Ann. Inst. Fourier (Grenoble) \textbf{52} (2002), no.~6,
  1785--1792. \MR{MR1952531 (2003k:32050)}

\bibitem{BHN03}
Judith Brinkschulte, C.~Denson Hill, and Mauro Nacinovich, \emph{The
  {P}oincar\'e lemma and local embeddability}, Boll. Unione Mat. Ital. Sez. B
  Artic. Ric. Mat. (8) \textbf{6} (2003), no.~2, 393--398. \MR{MR1988212
  (2004c:32073)}

\bibitem{C94}
David Catlin, \emph{Sufficient conditions for the extension of {CR}
  structures}, J. Geom. Anal. \textbf{4} (1994), no.~4, 467--538. \MR{MR1305993
  (95j:32028)}

\bibitem{E1970}
Ja.~M. {\`E}lia{\v{s}}berg, \emph{Singularities of folding type}, Izv. Akad.
  Nauk SSSR Ser. Mat. \textbf{34} (1970), 1110--1126. \MR{0278321 (43 \#4051)}

\bibitem{H88}
C.~Denson Hill, \emph{What is the notion of a complex manifold with a smooth
  boundary?}, Algebraic analysis, {V}ol.\ {I}, Academic Press, Boston, MA,
  1988, pp.~185--201. \MR{MR992454 (90e:32009)}

\bibitem{H91}
\bysame, \emph{Counterexamples to {N}ewlander-{N}irenberg up to the boundary},
  Several complex variables and complex geometry, {P}art 3 ({S}anta {C}ruz,
  {CA}, 1989), Proc. Sympos. Pure Math., vol.~52, Amer. Math. Soc., Providence,
  RI, 1991, pp.~191--197. \MR{MR1128593 (92m:32027)}

\bibitem{HN93}
C.~Denson Hill and Mauro Nacinovich, \emph{Embeddable {CR} manifolds with
  nonembeddable smooth boundary}, Boll. Un. Mat. Ital. A (7) \textbf{7} (1993),
  no.~3, 387--395. \MR{MR1249115 (95d:32019)}

\bibitem{HN99}
\bysame, \emph{Solvable {L}ie algebras and the embedding of {CR} manifolds},
  Boll. Unione Mat. Ital. Sez. B Artic. Ric. Mat. (8) \textbf{2} (1999), no.~1,
  121--126. \MR{MR1794546 (2001j:32039)}

\bibitem{HN06}
\bysame, \emph{On the failure of the {P}oincar\'e lemma for
  {$\overline\partial_M$}. {II}}, Math. Ann. \textbf{335} (2006), no.~1,
  193--219. \MR{MR2217688 (2006m:32043)}

\bibitem{Ho61}
Lars H{\"o}rmander, \emph{On existence of solutions of partial differential
  equations}, Partial differential equations and continuum mechanics, Univ. of
  Wisconsin Press, Madison, Wis., 1961, pp.~233--240. \MR{MR0124601 (23
  \#A1913)}

\bibitem{J86}
Howard Jacobowitz, \emph{Simple examples of nonrealizable {CR} hypersurfaces},
  Proc. Amer. Math. Soc. \textbf{98} (1986), no.~3, 467--468. \MR{MR857942
  (87k:32034)}

\bibitem{J88}
\bysame, \emph{Homogeneous solvability and {CR} structures}, Notas de Curso
  [Course Notes], vol.~25, Universidade Federal de Pernambuco Departamento de
  Matem\'atica, Recife, 1988. \MR{934571 (89f:35154)}

\bibitem{JT82}
Howard Jacobowitz and Fran{\c{c}}ois Tr{\`e}ves, \emph{Nonrealizable {CR}
  structures}, Invent. Math. \textbf{66} (1982), no.~2, 231--249.

\bibitem{JT83}
\bysame, \emph{Nowhere solvable homogeneous partial differential equations},
  Bull. Amer. Math. Soc. (N.S.) \textbf{8} (1983), no.~3, 467--469.

\bibitem{K67}
Walter Koppelman, \emph{The {C}auchy integral for functions of several complex
  variables}, Bull. Amer. Math. Soc. \textbf{73} (1967), 373--377. \MR{0209519
  (35 \#416)}

\bibitem{K82a}
Masatake Kuranishi, \emph{Strongly pseudoconvex {CR} structures over small
  balls. {I}. {A}n a priori estimate}, Ann. of Math. (2) \textbf{115} (1982),
  no.~3, 451--500. \MR{MR657236 (84h:32023a)}

\bibitem{K82b}
\bysame, \emph{Strongly pseudoconvex {CR} structures over small balls. {II}.
  {A} regularity theorem}, Ann. of Math. (2) \textbf{116} (1982), no.~1, 1--64.
  \MR{MR662117 (84h:32023b)}

\bibitem{K82c}
\bysame, \emph{Strongly pseudoconvex {CR} structures over small balls. {III}.
  {A}n embedding theorem}, Ann. of Math. (2) \textbf{116} (1982), no.~2,
  249--330. \MR{MR672837 (84h:32023c)}

\bibitem{hl56}
Hans Lewy, \emph{On the local character of the solutions of an atypical linear
  differential equation in three variables and a related theorem for regular
  functions of two complex variables}, Ann. of Math. (2) \textbf{64} (1956),
  514--522.

\bibitem{L57}
\bysame, \emph{An example of a smooth linear partial differential equation
  without solution}, Ann. of Math. (2) \textbf{66} (1957), 155--158.
  \MR{MR0088629 (19,551d)}

\bibitem{MM94}
Lan Ma and Joachim Michel, \emph{Regularity of local embeddings of strictly
  pseudoconvex {CR} structures}, J. Reine Angew. Math. \textbf{447} (1994),
  147--164. \MR{1263172 (95d:32017)}

\bibitem{mez93}
Abdelhamid Meziani, \emph{Perturbation of a class of {CR} structures of
  codimension larger than one}, J. Funct. Anal. \textbf{116} (1993), no.~1,
  225--244.

\bibitem{N84}
M.~Nacinovich, \emph{Poincar\'e lemma for tangential {C}auchy-{R}iemann
  complexes}, Math. Ann. \textbf{268} (1984), no.~4, 449--471. \MR{753407
  (86e:32025)}

\bibitem{NP11}
M.~Nacinovich and E.~Porten, \emph{$\mathcal{C}^{\infty}$-hypoellipticity and
  extension of $cr$ functions}, arXiv:1107.3374 (2011).

\bibitem{NN57}
A.~Newlander and L.~Nirenberg, \emph{Complex analytic coordinates in almost
  complex manifolds}, Ann. of Math. (2) \textbf{65} (1957), 391--404.
  \MR{0088770 (19,577a)}

\bibitem{Nir74}
Louis Nirenberg, \emph{A certain problem of {H}ans {L}ewy}, Uspehi Mat. Nauk
  \textbf{29} (1974), no.~2(176), 241--251, Translated from the English by Ju.
  V. Egorov, Collection of articles dedicated to the memory of Ivan
  Georgievi{\v{c}} Petrovski{\u\i} (1901--1973), I.

\bibitem{NT70}
Louis Nirenberg and Fran{\c{c}}ois Tr{\`e}ves, \emph{On local solvability of
  linear partial differential equations. {I}. {N}ecessary conditions}, Comm.
  Pure Appl. Math. \textbf{23} (1970), 1--38. \MR{MR0264470 (41 \#9064a)}

\bibitem{NT71}
\bysame, \emph{On local solvability of linear partial differential equations.
  {II}. {S}ufficient conditions}, Comm. Pure Appl. Math. \textbf{23} (1970),
  459--509. \MR{MR0264471 (41 \#9064b)}

\bibitem{PK87}
S.~I. Pinchuk and S.~V. Khasanov, \emph{Asymptotically holomorphic functions
  and their applications}, Mat. Sb. (N.S.) \textbf{134(176)} (1987), no.~4,
  546--555, 576. \MR{933702 (89f:32009)}

\bibitem{SW99}
H.~H. Schaefer and M.~P. Wolff, \emph{Topological vector spaces}, second ed.,
  Graduate Texts in Mathematics, vol.~3, Springer-Verlag, New York, 1999.
  \MR{MR1741419 (2000j:46001)}

\bibitem{Tr81}
F.~Tr{\`e}ves, \emph{Approximation and representation of functions and
  distributions annihilated by a system of complex vector fields}, \'Ecole
  Polytechnique Centre de Math\'ematiques, Palaiseau, 1981. \MR{716137
  (84k:58008)}

\bibitem{Tr92}
Fran{\c{c}}ois Tr{\`e}ves, \emph{Hypo-analytic structures}, Princeton
  Mathematical Series, vol.~40, Princeton University Press, Princeton, NJ,
  1992, Local theory. \MR{1200459 (94e:35014)}

\bibitem{Tu88}
A.~E. Tumanov, \emph{Extension of {CR}-functions into a wedge from a manifold
  of finite type}, Mat. Sb. (N.S.) \textbf{136(178)} (1988), no.~1, 128--139.
  \MR{945904 (89m:32027)}

\bibitem{Tu90}
\bysame, \emph{Extension of {CR}-functions into a wedge}, Mat. Sb. \textbf{181}
  (1990), no.~7, 951--964. \MR{1070489 (91f:32010)}

\end{thebibliography}
\end{document}